\documentclass[english]{amsart}

\usepackage{esint}
\usepackage[svgnames]{xcolor} 
\usepackage[colorlinks,citecolor=red,pagebackref,hypertexnames=false,breaklinks]{hyperref}
\usepackage{pgf,tikz}
\usepackage{pdfsync}

\usepackage{dsfont}
\usepackage{url}
\usepackage[utf8]{inputenc}
\usepackage[T1]{fontenc}
\usepackage{lmodern}
\usepackage{babel}
\usepackage{mathtools}  
\usepackage{amssymb}
\usepackage{lipsum}
\usepackage{mathrsfs}
\usepackage{color}

\newtheorem{theorem}{Theorem}[section]

\newtheorem{lemma}{Lemma}[section]
\newtheorem{definition}{Definition}[section]
\newtheorem{corollary}{Corollary}[section]

\newtheorem{remark}{Remark}[section]
\newtheorem{example}{Example}[section]

\numberwithin{equation}{section}

\title[Unique continuation ]{The  unique continuation property for second order evolution PDE{\tiny s}}

\author[Mourad Choulli]{Mourad Choulli}
\address{Universit\'e de Lorraine}
\email{mourad.choulli@univ-lorraine.fr}

\thanks{The author is supported by the grant ANR-17-CE40-0029 of the French National Research Agency ANR (project MultiOnde). }

\date{}

\begin{document}

\frenchspacing

\begin{abstract}
We present a simple and self-contained approach to establish the  unique continuation property for some classical evolution equations of second order in a cylindrical domain. We namely  discuss this property for wave, parabolic and Sch\"odinger operators with time-independent principal part. Our method is builds  on two-parameter Carleman inequalities combined with unique continuation across a pseudo-convex hypersurface with respect to the space variable. The most results we demonstrate in this work are more or less classical. Some of them are not stated exactly as in their original form.

\vskip .8cm
\end{abstract}

\subjclass[2010]{35A23, 35J15, 35K10, 35L10}

\keywords{Wave equation, parabolic equation, Schr\"odinger equation, elliptic equation, Carleman inequality, pseudo-convexity condition, non characteristic hypersurface, property of unique continuation, observability inequality.}

\maketitle

\tableofcontents

\section{Introduction}\label{introduction}

Let $D$ be a domain of $\mathbb{R}^d$, $d\ge 1$. We recall that $f\in C^\infty (D)$ is said real-analytic if its Taylor series around any arbitrary point of $D$ converges in a ball centered at this point. It is known that a real analytic function possesses the  unique continuation property which means that if $f$ vanishes in a nonempty open subset $D_0$ of $D$ then $f$ must vanishes identically (see for instance \cite[Theorem, page 65]{Jo1986Springer}). Functions satisfying this  unique continuation  property are also called quasi-analytic.

A classical result shows that a strong solution of an elliptic operator with smooth principal coefficients is quasi-analytic. In the present work we consider the analogue of this property for second order evolution PDEs. The right property in this context should be  the following: if a solution of an evolution equation of second order in the cylindrical domain $D\times (t_1,t_2)$ vanishes in $D_0\times (t_1,t_2)$, for some nonempty open subset $D_0$ of $D$, then this solution must vanishes identically. Unfortunately this property holds only in the parabolic case. Less optimal result still holds for the Sch\"odinger case. The worst case is that for wave equations for which we have only a weak version of unique continuation property, even for a large time interval.

We provide a simple and self contained approach to show the unique continuation property for wave, parabolic and Schr\"odinger equations. The approach we carry out is quite classical and it is based on two-parameter Carleman inequalities. The  unique continuation property in each case  is obtained as a consequence of the property of unique continuation across a non characteristic hypersurface satisfying in addition a pseudo-convexity condition in the case of wave and Schr\"odinger equations.

The core of our analysis consists in establishing two-parameter Carleman inequalities. We follow a classical scheme for obtaining these $L^2$-weighted energy estimates, essentially based on conjugating the original operator with a well chosen exponential function, splitting the resulting operator into its self-adjoint part and skew-adjoint part and finally making integrations by parts. The main assumption on the weight function is a pseudo-convexity condition with respect to the operator under consideration.  A systematic approach was considered by H\"ormander \cite[Section 28.2, page 234]{Ho2009Springer} for a general operator $P$ of an arbitrary order $m$ where the pseudo-convexity condition is expressed in term of the principal symbol of $P$. The method we develop in this work is more simple and does not appeal to fine analysis of PDEs and  our results have no pretension nor for generality neither for optimality.

It is worth remarking that splitting the conjugated operator into self-adjoint  and skew-adjoint parts is not the best possible way to get two-parameter Carleman inequalities for elliptic and parabolic operators. There is a particular way to split the conjugated operator into two parts. However this particular decomposition is not applicable for wave and Schr\"odinger equations. But this is not really surprising since solutions of wave and Schr\"odinger equations do not enjoy the same regularity properties of solutions of elliptic and parabolic equations.

We choose to start with the more subtle case corresponding to the wave equation. Since most calculations to get Carleman inequalities are common for different type of equations, Carleman inequalities for parabolic and Schr\"odinger equations are obtained by making some modifications in the proof of the Carleman inequality for the wave equation. We also added a short section for the elliptic case whose analysis is almost similar to that of the parabolic case.

One can find in the literature two-parameter Carleman inequalities with degenerate weight function. We refer for instance to \cite{Bo2017MCRF, Ch2009Springer, CY2017Arxiv, FG2006Sicon, FI1996Seoul,LL2012COCV} for parabolic operators and \cite{BP2002CRAS,BP2002IP,BP2007IP, MOR2008CRAS,MOR2008IP} for Schr\"odinger operators. 

A Carleman inequality for wave equations on compact Riemannian manifold can be found in \cite{BY2017Springer} and quite recently Huang \cite{Hu2020Springer} proved a Carleman inequality for a general wave operators with time-dependent principal part. The interested reader is referred to \cite{FLZ2019Springer} for a unified approach  establishing Carleman inequalities for second order PDEs and their applications to control theory and inverse problems.

We point out that it is possible to derive Carleman inequalities with single parameter in the weight function. We refer for instance to \cite{Ho1976Springer,Ho2009Springer,LTZ2004,Zh2000,Zh2001} and references therein for more details. Introducing a second parameter in the Carleman weight is a simple way to guarantee pseudo-convexity condition as it is remarked in \cite{Ho1976Springer}. It certainly contributed to the development of Carleman inequalities with two-parameter. Two-parameter Carleman inequalities appear to be more flexible then one-parameter Carleman inequalities since we have to our disposal two parameters. Of course, one-parameter Carleman inequalities can be deduced from two-parameter Carleman inequalities (see for instance Theorem  \ref{CarlemanTheoremWaveLocal} for the case of the wave equation).

The unique continuation property  for elliptic and parabolic operators with unbounded lower order coefficients was obtained in \cite{SS1980CRAS,SS1982JDE,SS1983CRAS,SS1987JDE}. The analysis in these references combine both classical tools used for establishing the property of unique continuation together with interpolation inequalities. Uniqueness and non-uniqueness for general operators were discussed in \cite{Al1983AM,AB1995MathZ,Zu1983Birkhauser} (see also the references therein). We also mention \cite{Is2017Springer,Ni1957CPAM} as additional references on uniqueness of Cauchy problems.

We also discuss briefly observability inequalities which can be seen as the quantification of unique continuation for the Cauchy problem associated to IBVPs. We refer to \cite{LL2019JEMS} for general observability inequalities for wave and Schr\"odinger equations with arbitrary interior or boundary observation region. The reader can find in this work a detailed introduction to explain the main steps to get the property of unique continuation for an intermediate case between Holmgren's analytic case and H\"ormander's general case for operators with partially analytic coefficients.

A more difficult problem consists in quantifying the property of unique continuation from an interior subdomain or the Cauchy data on a sub-boundary. The elliptic case is now almost completely solved with optimal results for $C^{1,\alpha}$-solutions and $C^{0,1}$-domains \cite{Ch2016Springer} or $H^2$-solutions and $C^{1,1}$-domains \cite{Bo2010M2AN}. A non optimal result for $H^2$-solutions and $C^{0,1}$-domains was obtained in \cite{Ch2020BAMS}. These kind of results can be obtained by a method based on three-sphere inequality which is deduced itself from a Carleman inequality. The case of parabolic and wave equations is extremely more difficult than  the elliptic case. Concerning parabolic equations, a first result was obtained in \cite{Bo2017MCRF} with Cauchy data in a particular subboundary. This result is based on a global Carleman inequality. The general case was tackled in \cite{CY2017Arxiv} where a non optimal result was established using a three-cylinder inequality. A result for the wave equation was recently proved in \cite{BC2019Arxiv}. This result was obtained via Fourier-Bros-Iagolnitzer transform allowing to transfer the quantification of unique continuation of an elliptic equation to that of the wave equation.

\section{Preliminaries}

\subsection{Main notations and assumptions}

Throughout $\Omega$ is bounded Lipschitz domain of $\mathbb{R}^n$, $n\ge 2$, with boundary $\Gamma$. The unit normal exterior vector field on $\Gamma$ is denoted by $\nu$. 

We set $Q=\Omega \times (t_1,t_2)$ and $\Sigma =\Gamma\times (t_1,t_2)$, where $t_1,t_2\in \mathbb{R}$ are fixed so that $t_1<t_2$.

$A=(a_{k\ell})$ will denote a symmetric matrix with coefficients $a_{k\ell}\in C^{2,1}(\overline{\Omega})$, $1\le k,\ell\le n$, and there exist two constants $\mathfrak{m}>0$ and $\varkappa\ge 1$ so that
\[
\varkappa ^{-1}|\xi| ^2\le \sum_{k,\ell=1}^na_{k\ell}(x)\xi_\ell  \xi _k \le \varkappa |\xi|^2,\quad x\in \Omega ,\; \xi\in \mathbb{R}^n,
\]
and
\[
\|A\|_{C^{2,1}(\overline{\Omega} ;\mathbb{R}^{n\times n})}\le \mathfrak{m}.
\]
The set of such matrices will denoted in the rest of this text by $\mathscr{M}(\Omega ,\varkappa ,\mathfrak{m})$.

It is worth mentioning that, according to Rademacher's theorem,  $C^{2,1}(\overline{\Omega})$ is continuously embedded in $W^{3,\infty}(\Omega )$.

For $\xi,\eta \in \mathbb{C}^n$, $(\xi|\eta)$ and $\xi\otimes \eta$ are defined as usual respectively by
\[
(\xi|\eta)=\sum_{k=1}^n \xi_k\overline{\eta_k} \quad \mbox{and}\quad \xi\otimes \eta =(\xi_k\eta_\ell)_{1\le k,\ell\le n}.
\]

The Jacobian matrix and the Hessian matrix are denoted respectively by
\begin{align*}
&U'=(\partial_\ell U_k),\quad U\in H^1(\Omega ;\mathbb{C}^n),
\\
&\nabla^2 u=(\partial_{k\ell}^2u),\quad u\in H^2(\Omega ;\mathbb{C}).
\end{align*}

The norm (resp. the scalar product) of a Banach (resp. a Hilbert) space $E$  is always denoted by $\|\cdot\|_E$ (resp. $\langle\cdot|\cdot\rangle_E$).

We recall that the anisotropic Sobolev space $H^{2,1}(Q)$ is defined as follows
\[
H^{2,1}(Q)=L^2((t_1,t_2);H^2(\Omega))\cap H^1((t_1,t_2);L^2(\Omega)).
\]

The following notations will be useful  in the sequel, where $S=\Omega$ or $S=\Gamma$,
\begin{align*}
& \nabla_Au(x)=A(x)^{1/2}\nabla u(x), \quad u\in H^1(\Omega;\mathbb{C}),
\\
&\mbox{div}_AU(x)=\mbox{div}(A(x)^{1/2}U(x)),\quad U \in H^1(S ;\mathbb{C}^n),
\\
&\Delta_Au(x)=\mbox{div}_A\nabla_Au(x)=\mbox{div}(A(x)\nabla u(x)),\quad u\in H^2(\Omega;\mathbb{C}),
\\
&(U|V)_A(x)=(A(x)U(x)|V(x))=(U(x)|A(x)V(x)),\quad U,V \in L^2(S ;\mathbb{C}^n),
\\
&|U|_A(x)=[(U|U)_A(x)]^{1/2},\quad U\in L^2(S ;\mathbb{C}^n).
\end{align*}
We readily obtain from the above definitions the following identity
\[
(\nabla_Au|\nabla_Av) =(\nabla u|\nabla v)_A,\quad u,v\in H^1(\Omega;\mathbb{C}).
\]
Furthermore, the Green type formula 
\begin{align}
\int_\Omega \Delta_Au \overline{v}dx&=-\int_\Omega (\nabla_Au|\nabla_Av)dx+\int_\Gamma (\nabla u|\nu)_Avd\sigma \label{Gf}
\\
&=-\int_\Omega (\nabla u|\nabla v)_Adx+\int_\Gamma (\nabla u|\nu)_Avd\sigma, \nonumber 
\end{align}
holds for any $u\in H^2(\Omega;\mathbb{C} )$ and $v\in H^1(\Omega ;\mathbb{C})$.

We need also to introduce the following notations 
\begin{align*}
&\mathcal{L}_{A,0}^e=\Delta _A\hskip 1.1cm \mbox{(elliptic operator)},
\\
&\mathcal{L}_{A,0}^p=\Delta _A-\partial_t\quad \; \mbox{(parabolic operator)},
\\
&\mathcal{L}_{A,0}^w=\Delta _A-\partial_t^2\quad \mbox{(wave operator)},
\\
&\mathcal{L}_{A,0}^s=\Delta _A+i\partial_t\quad \mbox{(Schr\"odinger operator)}.
\end{align*}

The  $n\times n$ identity matrix will denoted by $\mathbf{I}$.

We shall use for notational convenience the following notation
\[
[h]_{t=t_1}^{t_2}=h(\cdot ,t_2)-h(\cdot ,t_1),\quad h\in H^1((t_1,t_2);L^2(\Omega )).
\]

Finally, we equip $\partial Q$ with following measure
\[
d\mu(x,t) =\mathds{1}_{\Gamma \times (t_1,t_2)}(x,t)d\sigma (x)dt +\mathds{1}_{\Omega \times \{t_1,t_2\}}(x,t)dx\delta_t,
\]
where $d\sigma (x)$ is the Lebesgue measure on $\Gamma$ and $\delta_t$ is the Dirac measure at $t$.

We used above $\mathds{1}_X$ to denote the characteristic function of the measurable set $X$:
\[
\mathds{1}_X(x)=\left\{ \begin{array}{ll} 1\quad &\mbox{if}\; x\in X,\\ 0\quad &\mbox{if}\; x\not\in X.\end{array} \right.
\]

\subsection{Pseudo-convexity condition}

Define, for $A\in \mathscr{M}(\Omega ,\varkappa ,\mathfrak{m})$, 
\[
\Lambda_{k\ell}^m(A)(x)=-\sum_{p=1}^n\partial_pa_{k\ell}(x)a_{pm}(x)+2\sum_{p=1}^na_{kp}(x)\partial_pa_{\ell m}(x),
\]
where $x\in \overline{\Omega}$ and $1\le k,\ell,m\le n$.

We associate to $h\in C^1(\overline{\Omega})$ the matrix $\Upsilon_A (h)$ given  by
\[
(\Upsilon_A (h))_{k\ell}(x)=\sum_{m=1}^n\Lambda_{k\ell}^m(A)(x)\partial_m h (x),\quad x\in \overline{\Omega} .
\]

Note that $\Upsilon_A (h)$ is not necessarily symmetric. 

Inspired by the definition introduced in \cite[Section 28.2, page 234]{Ho2009Springer} we consider the following one:

\begin{definition}\label{definitionPC}
We say that $h\in C^2(\overline{\Omega})$ is $A$-pseudo-convex with constant $\kappa >0$ in $\Omega$  if $\nabla h(x)\ne 0$ for any $x\in \overline{\Omega}$ and if 
\[
(\Theta_A(h)(x)\xi|\xi)\ge \kappa |\xi|^2,\; x\in \overline{\Omega},\; \xi \in \mathbb{R}^n,
\]
where
\[
\Theta_A(h)=2A\nabla^2hA+\Upsilon_A(h).
\]
\end{definition}

It worth noticing that $A\rightarrow \Theta_A$ is positively homogenous of degree two: 
\[
\Theta_{\lambda A}=\lambda^2\Theta_A,\quad  \lambda>0.
\]

Since $\Theta_{\mathbf{I}}(h)=2\nabla^2h$, $h$ is $\mathbf{I}$-pseudo-convex in $\Omega$  if $\nabla h(x)\ne 0$ and $\nabla^2h(x)$ is positive definite for any $x\in \overline{\Omega}$. In other words, when $A=\mathbf{I}$ pseudo-convexity is reduced to local strict convexity.

\subsection{Carleman weights} 

It will be convenient to define the notion of Carleman weight for different kind of operators we are interested in. In the rest of this paper $\psi=\psi (x,t)$ is a function of the form
\[
\psi(x,t)=\psi_0(x)+\psi_1(t),\quad (x,t)\in Q,
\]
and $\phi =e^{\lambda \psi}$, $\lambda >0$.

\begin{definition}\label{definitionCW}
(a) Let $0\le \psi_0 \in C^4(\overline{\Omega})$. We say that $\phi_0=e^{\lambda \psi_0}$, $\lambda >0$, is a weight function for the elliptic operator $\mathcal{L}_{A,0}^e$ if $\nabla \psi_0(x)\ne 0$ for any $x\in \overline{\Omega}$.
\\
(b) If $0\le \psi \in C^4(\overline{Q})$ and $\nabla \psi_0(x)\ne 0$, for any $x\in \overline{\Omega}$, we say that $\phi$ is a weight function for the parabolic operator $\mathcal{L}_{A,0}^p$.
\\
(c) Assume that $0\le \psi \in C^4(\overline{Q})$. Then $\phi$ is said a weight function for the Schr\"odinger operator $\mathcal{L}_{A,0}^s$ if $\psi_0$ is A-pseudo-convex in $\Omega$.
\\
(d) We say that $\phi$, with $0\le \psi \in C^4(\overline{Q})$, is a weight function for the wave operator $\mathcal{L}_{A,0}^w$ if $\psi_0$ is A-pseudo-convex with constant $\kappa >0$ in $\Omega$ and if, in addition, the following two conditions hold:
\begin{align}
&\min_{\overline{Q}}\left[|\nabla \psi_0|_A^2-(\partial_t\psi_1 )^2\right]^2>0,\label{int1}
\\
&|\partial_t^2\psi_1|\le \varkappa^{-1}\kappa/4\label{int2}.
\end{align}
\end{definition}

\begin{example}\label{example1}
{\rm
Fix $t_0\in \mathbb{R}$, $x_0\in \mathbb{R}^n\setminus \overline{\Omega}$ and set, for $\gamma \in \mathbb{R}$,
\[
\psi (x,t)=\left[|x-x_0|^2+\gamma (t+t_0)^2\right]/2+C, \quad (x,t)\in \overline{Q},
\]
where the constant $C$ is chosen sufficiently large in order to guarantee that $\psi \ge 0$. In that case
\[
(\Upsilon_A(\psi_0))_{\alpha \beta}=-\sum_{k,\ell=1}^n\partial_ka_{\alpha\beta}a_{k\ell}(x_\ell -x_{0,\ell})+2\sum_{k,\ell=1}^na_{\alpha k}\partial_ka_{\beta\ell}(x_\ell -x_{0,\ell}) .
\]

Let us first discuss A-pseudo-convexity condition of $\psi_0$ in different cases.

(i) Assume that $\Omega =B(0,r)$ and $x_0\in B(0,2r)\setminus \overline{B}(0,r)$. We can then choose $r$ sufficiently small in such a way that
\[
(\Upsilon_A(\psi_0)\xi |\xi)\ge -\varkappa^2 |\xi|^2,
\]
from which we deduce that
\[
(\Theta_A(\psi_0)\xi |\xi)\ge \varkappa ^2|\xi|^2.
\]

(ii) As the mapping
\[
A \in C^{2,1}(\overline{\Omega} ,\mathbb{R}^{n\times n})\mapsto \Upsilon_A(\psi_0)\in C^{1,1}(\overline{\Omega} ,\mathbb{R}^{n\times n})
\]
is continuous in a neighborhood of $\mathbf{I}$ and $\Upsilon_{\mathbf{I}}(\psi_0)=0$, we conclude that there exists $\mathcal{N}$, a neighborhood of $\mathbf{I}$ in $C^{2,1}(\overline{\Omega} ,\mathbb{R}^{n\times n})$,  so that, for any $A\in \mathcal{N}$, we have
\begin{align*}
& (A(x)\xi |\xi )\ge |\xi|^2/2,\quad x\in \overline{\Omega} ,\; \xi \in \mathbb{R}^n,
\\
&(\Upsilon_A(\psi_0)\xi |\xi)\ge -|\xi|^2/4,\quad x\in \overline{\Omega} ,\; \xi \in \mathbb{R}^n.
\end{align*}
Whence
\[
(\Theta_A(\psi_0)\xi |\xi)\ge |\xi|^2/4,\quad x\in \overline{\Omega} ,\; \xi \in \mathbb{R}^n,
\]
provided that $A\in \mathcal{N}$.

(iii) Consider the particular case in which $A=a\mathbf{I}$ with $a\in C^{2,1}(\overline{\Omega})$ satisfying $a\ge \varkappa$. Simple computations then yield
\[
\Upsilon_A(\psi_0)= -a\left(\nabla a |x-x_0\right)\mathbf{I}+2a\nabla  a\otimes (x-x_0).
\]
In consequence
\[
(\Theta_A(\psi_0)\xi |\xi)\ge \varkappa \left(2\varkappa-3\left| \nabla  a\right||x-x_0|\right) |\xi|^2.
\]
Hence a condition guaranteeing that $\psi_0$ is $a\mathbf{I}$-pseudo-convex is
\[
\left| \nabla a\right||x-x_0|<2\varkappa/3.
\]
This condition is achieved for instance if $\Omega$ has sufficiently small diameter and $x_0$ is close to $\Omega$ or else $|\nabla a|$ is small enough.

Next, we discuss a bound on $\gamma$ for which \eqref{int1} and \eqref{int2} hold simultaneously. If $d=\mbox{dist}(x_0,\overline{\Omega})$ $(>0)$ and $\sigma =\|t+t_0\|_{L^\infty ((t_1,t_2))}$, then \eqref{int1} is satisfied whenever  
\[
|\nabla \psi|^2-(\partial_t \psi )^2\ge d^2- \sigma^2 \gamma^2  >0.
\]
 As $\partial_t^2\psi =\gamma$, we see that both \eqref{int1} and \eqref{int2} are satisfied when
\[
0<|\gamma| <\min \left[d/\sigma ,\kappa/(4\varkappa)\right].
\]
}
\end{example}

\subsection{Pseudo-convex hypersurface}\label{subsectionPCH}

We begin by a lemma concerning the action of an orthogonal transformation on $A\in \mathscr{M}(\Omega ,\varkappa,\mathfrak{m})$. If $\mathcal{O}$ is an orthogonal transformation and $A\in \mathscr{M}(\Omega ,\varkappa,\mathfrak{m})$, we set $A_{\mathcal{O}}(y)=\mathcal{O}A(\mathcal{O}^ty)\mathcal{O}^t$. Here $\mathcal{O}^t$ denotes the transposed matrix of $\mathcal{O}$.

The proof of the following lemma is straightforward.

\begin{lemma}\label{lemmaPCH0}
Let $A\in \mathscr{M}(\Omega ,\varkappa,\mathfrak{m})$ and $\mathcal{O}$ an orthogonal transformation. Then $A_{\mathcal{O}}\in \mathscr{M}(\mathcal{O}\Omega,\varkappa ,\mathfrak{m}')$, where $\mathfrak{m}'=\mathfrak{m}'(n,\mathfrak{m})>0$ is a constant.
\end{lemma}

The gradient with respect to the variable $x'\in \mathbb{R}^{n-1}$ or $y'\in \mathbb{R}^{n-1}$ is denoted hereafter by $\nabla '$.

Let $\theta$ be a $C^{3,1}$-function defined in a neighborhood $\mathcal{U}$ of $\tilde{x}$ in  $\Omega$ with $\nabla \theta (\tilde{x})\ne 0$. Consider then the hypersurface 
\[
H=\{x\in \mathcal{U};\; \theta(x)=\theta(\tilde{x})\}.
\]
Making a translation and change of coordinates we may assume that $\tilde{x}=0$, $\theta(\tilde{x})=0$, $\nabla '\theta (0)=0$ and $\partial_n\theta (0)\ne 0$. With the help of the implicit function theorem $\theta (x)=0$ near $0$ may rewritten as $x_n=\vartheta (x')$ with $\vartheta (0)=0$ and $\nabla '\vartheta (0) =0$. 

Let $\hat{A}$ be the matrix obtained after this transformations. According to Lemma \ref{lemmaPCH0}, $\hat{A}\in \mathscr{M}(\mathcal{O}(\mathcal{U}+\tilde{x}),\varkappa ,\mathfrak{m}')$, $\mathfrak{m}'=\mathfrak{m}'(n,\mathfrak{m})>0$ is a constant, where $\mathcal{O}$ is the orthogonal transformation corresponding to the above change of coordinates. Also, note that $\hat{\mathbf{I}}=\mathbf{I}$.

Consider, in some neighborhood  of $0$, the mapping
\begin{equation}\label{PCH0}
\varphi =\varphi_H :(x',x_n)\in \omega \mapsto (y',y_n)=(x',x_n-\vartheta (x')+|x'|^2).
\end{equation}
Elementary calculations yield
\[
\varphi'(x',x_n)=\left(
\begin{array} {cccc}
1 &\ldots &0 &0
\\
\vdots &\ddots &0 &\vdots
\\
0 &\ldots &1 &0
\\
g_1(x') &\ldots &g_{n-1}(x') &1
\end{array}
\right)
\]
with $g_k(x')=-\partial_k\vartheta (x')+2x_k'$,  $0\le k\le n-1$. Whence
\[
(\varphi '(x',x_n)\xi|\xi)= |\xi'|^2+(-\nabla'\vartheta (x')+2x'|\xi ')\xi_n +\xi_n^2,\quad \xi=(\xi',\xi_n) \in \mathbb{R}^n.
\]
Since
\begin{align*}
|(-\nabla'\vartheta (x')+2x'|\xi ')\xi_n|&\le |(-\nabla'\vartheta (x')+2x'|\xi ')|^2/2+\xi_n^2/2
\\
&\le |-\nabla'\vartheta (x')+2x'|^2|\xi '|^2/2+2\xi_n^2/2
\\
&\le \left(|\nabla'\vartheta (x')|^2+4|x'|^2\right)|\xi '|^2+\xi_n^2/2,
\end{align*}
there exists a neighborhood $\omega$ of $0$, only depending of $\vartheta$, so that
\[
|(-\nabla'\vartheta (x')+2x'|\xi ')\xi_n|\le |\xi '|^2/2+\xi_n^2/2=|\xi|^2/2.
\]
In consequence
\begin{equation}\label{PCH1}
(\varphi '(x',x_n)\xi|\xi)\ge |\xi|^2/2,\quad x=(x',x_n)\in \omega ,\; \xi \in \mathbb{R}^n.
\end{equation}
Whence \eqref{PCH1} together with Cauchy-Schwarz's inequality yield
\begin{equation}\label{PCH1.1}
|(\varphi')^t(x',x_n)\xi |^2\ge |\xi|^2/2,\quad x=(x',x_n)\in \omega,\; \xi \in \mathbb{R}^n.
\end{equation}
Let $\tilde{\omega}=\varphi(\omega)$ (hence $\varphi$ is a diffeomorphism from $\omega$ onto $\tilde{\omega}$) and define
\begin{equation}\label{PCH2}
\tilde{A} (y)=\varphi'\left(\varphi ^{-1}(y)\right)\hat{A}\left(\varphi ^{-1}(y)\right)(\varphi')^t\left(\varphi ^{-1}(y)\right),\quad y\in \tilde{\omega}.
\end{equation}
In light of \eqref{PCH1.1} we obtain
\[
(\tilde{A}(y)\xi|\xi)\ge \varkappa |(\varphi')^t\left(\varphi ^{-1}(y)\right)\xi|^2\ge (\varkappa/2)|\xi|^2,\quad y\in \tilde{\omega}.
\]
Also, by straightforward computations we get
\[
\|\tilde{A}\|_{C^{2,1}(\overline{\Omega} ;\mathbb{R}^{n\times n})}\le \tilde{\mathfrak{m}},
\]
with $\tilde{\mathfrak{m}}$ only depending of $n$, $\mathfrak{m}$ and $\vartheta$.

We observe that the role of $\varphi$ is to transform the hypersurface $\{x_n=\vartheta(x')\}$ in a neighborhood of the origin into the convex hypersurface $\{y_n=|y'|^2\}$ in another neighborhood of the origin.

Define $\tilde{\psi}_0$ as follows
\[
\tilde{\psi}_0(y)=(y_n-1)^2+|y'|^2.
\]
The matrix $\tilde{A}$ appearing in \eqref{PCH2} is denoted hereafter by $A_H$. The following definition is motivated by the classical procedure used to establish the unique continuation property of an elliptic operator across the convex hypersurface $\{y_n=|y'|^2\}$.
\begin{definition}\label{definitionPCH1}
We say that the hypersurface $H$ is $A$-pseudo-convex if $\tilde{\psi}_0$ is $A_H$-pseudo-convex in $\tilde{\omega}$. 
\end{definition}

\begin{lemma}\label{lemmaPCH}
(a) There exists $\mathcal{N}$, a neighborhood of $\mathbf{I}$ in $C^{2,1}(\overline{\Omega} ;\mathbb{R}^{n\times n})$, so that, for any $A\in \mathcal{N}$, $H$ is $A$-pseudo-convex.
\\
(b) There exists $\mathcal{N}_0$, a neighborhood of $\mathbf{I}$ in $C^{2,1}(\overline{\Omega} ;\mathbb{R}^{n\times n})$, so that, for any $A\in \mathcal{N}_0$ and any orthogonal transformation $\mathcal{O}$, we have $A_\mathcal{O}\in \mathcal{N}$.
\end{lemma}

\begin{proof}
(a) Let us first discuss the  case where $A=\mathbf{I}$. Note that it is not hard to check that
\[
\varphi^{-1}(y',y_n)=\left(y',y_n+\vartheta (y')-|y'|^2\right)
\]
and
\[
\mathbf{I}_H(y)=\tilde{\mathbf{I}} (y)=(\tilde{a}_{ij}(y'))=\left(
\begin{array} {cccc}
1 &\ldots &0 &\tilde{g}_1(y')
\\
\vdots &\ddots &0 &\vdots
\\
0 &\ldots &1 &\tilde{g}_{n-1}(y')
\\
\tilde{g}_1(y') &\ldots &\tilde{g}_{n-1}(y') &\tilde{g}_n(y')
\end{array}
\right),
\]
with $\tilde{g}_k(y')=\partial_k\vartheta (y')-2y_k'$,  $0\le k\le n-1$ and $\tilde{g}_n=|\nabla \vartheta(y')-2y'|^2+1$.

We have clearly $\tilde{A} (0)=\mathbf{I}$ and, for $1\le p\le n-1$,
\[
\partial_p\tilde{a}_{k\ell}(y')=
\left\{
\begin{array}{ll}
0,\quad  1\le k,\ell\le n-1,
\\
\\
\partial_{pk}^2\vartheta (y')-2\delta_{pk}, \quad 1\le k\le n-1,\; \ell=n,
\\
\\
2\sum_{\alpha=1}^{n-1}(\partial_\alpha \vartheta(y')-2y_\alpha)(\partial_{p\alpha}^2 \vartheta(y')-2\delta_{p\alpha}) ,\quad k=n,\; \ell=n.
\end{array}
\right.
\]
Therefore
\[
\partial_p\tilde{a}_{k\ell}(0)=
\left\{
\begin{array}{ll}
0, &\quad 1\le k,\ell\le n-1,
\\
\\
\partial_{pk}^2\vartheta (0)-2\delta_{pk} &1\le k\le n-1,\; \ell=n,
\\
\\
0 & k=n,\; \ell=n.
\end{array}
\right.
\]
Let $\tilde{\Lambda}_{k\ell}^m$ given by
\[
\tilde{\Lambda}_{k,\ell}^m(y)=-\sum_{p=1}^n\partial_p\tilde{a}_{k\ell}(y)\tilde{a}_{pm}(y)+2\sum_{p=1}^n\tilde{a}_{kp}(y)\partial_p\tilde{a}_{\ell m}(y)
\]
and define  $\tilde{\Upsilon}(y)=(\tilde{\Upsilon}_{k\ell}(y))$ as follows
\[
\tilde{\Upsilon}_{k\ell}(y)=\sum_{m=1}^n\tilde{\Lambda}_{k\ell}^m(y)\partial_m \tilde{\psi}_0(y) .
\]
It is then straightforward to check that
\[
\tilde{\Upsilon}_{k\ell}(0)=\sum_{m=1}^n\Lambda_{k,\ell}^m(0)\partial_m \tilde{\psi}_0(0)=-2\Lambda_{k,\ell}^n(0)=0,\quad 0\le k,\ell \le n.
\]
Since
\[
\tilde{\Theta}(y)=\Theta_{\mathbf{I}}(\tilde{\psi}_0)(y)= 2\nabla ^2\tilde{\psi}_0(y)+\tilde{\Upsilon}(y),
\]
we get
\[
\tilde{\Theta}(0)= 4\mathbf{I},
\]
and hence
\[
(\tilde{\Theta}(0)\xi|\xi)\ge 4|\xi|^2,\quad \xi \in \mathbb{R}^n.
\]
Continuity argument, first with respect to $y$ and then with respect to $A$, shows that, by reducing $\tilde{\omega}$ if necessary,
\[
(\Theta_{\tilde{A}}(\psi_0)(y)\xi|\xi)\ge 2|\xi|^2,\quad y\in \tilde{\omega},\; \xi \in \mathbb{R}^n.
\]
(b) Immediate from Lemma \ref{lemmaPCH0}.
\end{proof}

\section{The wave equation}

\subsection{Carleman inequality}

In this subsection $\psi(x,t)=\psi_0(x)+\psi_1(t)$ is a weight function for the wave operator $\mathcal{L}_{A,0}^w$ with $A$-pseudo-convexity constant $\kappa>0$. We set
\[
\delta = \min_{\overline{Q}}\left[|\nabla \psi|_A^2-(\partial_t\psi )^2\right]^2\;  (>0)
\]
and  $\phi=e^{\lambda \psi}$.

We use for notational convenience $\mathfrak{d}=(\Omega, t_1, t_2,\varkappa, \mathfrak{m}, \kappa,\delta ,\mathfrak{b})$ with $\mathfrak{b}\ge \|\psi\|_{C^4(\overline{Q})}$, and $\mathbf{D}_A=(\nabla _A\, \cdot , \partial_t\, \cdot)$.

\begin{theorem}\label{CarlemanTheorem1}
There exist three constants $\aleph=\aleph(\mathfrak{d})$, $\lambda^\ast=\lambda^\ast(\mathfrak{d})$ and $\tau^\ast =\tau^\ast(\mathfrak{d})$ so that
\begin{align}
&\aleph\int_Qe^{2\tau \phi}\left[\tau^3\lambda ^4\phi^3 u^2+\tau \lambda  \phi |\mathbf{D}_A u|^2\right]dxdt \label{C1.0}
\\
&\hskip1cm \le \int_Qe^{2\tau \phi}\left(\mathcal{L}_{A,0}^wu\right)^2dxdt +\int_{\partial Q} e^{2\tau \phi}\left[\tau^3\lambda ^3\phi^3u^2+\tau \lambda \phi |\mathbf{D}_A u|^2\right]d\mu ,\nonumber
\end{align}
for any $\lambda \ge \lambda^\ast$, $\tau \ge \tau^\ast$ and $u\in H^2(Q,\mathbb{R})$.
\end{theorem}

\begin{proof}
In this proof, $\aleph_j$, $\lambda_j$ and $\tau_j$, $j=0,1,\ldots$, are positive generic constants only depending on $\mathfrak{d}$.

Set $\Phi=e^{-\tau \phi}$ with $\tau >0$. Elementary computations then give
\begin{align*}
&\partial_k\Phi =-\tau \partial_k\phi \Phi  ,
\\
&\partial_{k\ell}\Phi=\left(-\tau \partial^2_{k\ell}\phi +\tau^2\partial_k\phi\partial_\ell\phi\right)\Phi,
\\
&\partial_t\Phi= -\tau \partial_t\phi\Phi  ,
\\
&\partial_t^2\Phi=\left(-\tau  \partial_t^2 \phi  +\tau^2(\partial_t\phi)^2\right)\Phi.
\end{align*}
The preceding two first formulas can be rewritten in the following form
\begin{align*}
&\nabla \Phi =-\tau \Phi \nabla \phi,
\\
&\nabla ^2\Phi =\Phi \left(-\tau \nabla ^2\phi +\tau^2\nabla \phi \otimes \nabla \phi\right).
\end{align*}
For $w\in H^2 (Q;\mathbb{R})$, we obtain
\[
\Phi^{-1}\Delta_A(\Phi w)=\Delta_Aw -2\tau (\nabla w|\nabla \phi)_A+\left[\tau^2  |\nabla \phi|_A^2-\tau\Delta_A\phi\right]w.
\]
Also,
\[
\Phi^{-1}\partial_t^2(\Phi w)=\partial_t^2w-2\tau \partial_t\phi \partial_tw
+\left(-\tau  \partial_t^2 \phi  +\tau^2(\partial_t\phi)^2\right) w.
\]
We decompose $L=\Phi^{-1}\mathcal{L}_{A,0}^w\Phi $ into its self-adjoint part and skew-adjoint part:
\[
L=L_++L_-.
\]
Simple calculations show that
\begin{align*}
&L_+w=\Delta_A w-\partial_t^2w+aw,
\\
&L_-w= (B|\nabla w)+d \partial_tw +bw,
\end{align*}
with
\begin{align*}
&a(x,t)=\tau^2\left(|\nabla \phi|_A^2-(\partial_t\phi)^2\right),
\\
&b(x,t)=-\tau \left(\Delta_A  \phi -\partial_t^2\phi\right),
\\
&B=-2\tau  A\nabla\phi,
\\
&d=2\tau\partial_t\phi .
\end{align*}
We have
\begin{equation}\label{ani1}
\langle L_+w|L_-w\rangle_{L^2(Q)}=\sum_{j=1}^9 I_j,
\end{equation}
where
\begin{align*}
&I_1=\int_Q\Delta_A w(\nabla w|B) dxdt,
\\
&I_2=\int_Q\Delta_A wd\partial_twdxdt,
\\
&I_3=\int_Q\Delta_A wbwdxdt,
\\
&I_4=-\int_Q\partial_t^2w(\nabla w|B) dxdt,
\\
&I_5=-\int_Q\partial_t^2wd\partial_tw dxdt,
\\
&I_6=-\int_Q\partial_t^2wbw dxdt,
\\
&I_7=\int_Qaw(\nabla w|B) dxdt,
\\
&I_8=\int_Qadw\partial_twdxdt,
\\
&I_9=\int_Qabw^2dxdt.
\end{align*}
The most integrations by parts, with respect to the space variable,  we use in  this proof are often based on Green's formula \eqref{Gf}. A first integration by parts then yields 
\begin{align*}
I_1=\int_Q\Delta_Aw (\nabla w|B) dxdt =-\int_Q (\nabla w&|\nabla (\nabla w|B))_Adxdt
\\
&+\int_\Sigma (\nabla w|\nu )_A(\nabla w|B) d\sigma dt.
\end{align*}
Whence
\begin{equation}\label{ani2}
I_1=-\int_Q ([\nabla ^2wB+(B')^t\nabla w]|\nabla w)_Adxdt +\int_\Sigma (\nabla w|\nu )_A(\nabla w|B) d\sigma dt.
\end{equation}
Now as
\begin{align*}
\int_Q \partial_{\ell j}^2w B_ja_{\ell k}\partial_k wdxdt=-\int_Q\partial_\ell w&\partial_j(B_ja_{\ell k})\partial_kwdxdt -\int_Q\partial_\ell wB_ja_{\ell k}\partial_{kj}^2wdxdt
\\
&+\int_\Sigma \partial_\ell wB_j\nu_ja_{\ell k}\partial_k wd\sigma dt,
\end{align*}
we find
\begin{equation}\label{ani3}
2\int_Q (\nabla ^2wB|\nabla w)_Adxdt=-\int_Q(C\nabla w|\nabla w)+\int_\Sigma (B|\nu)|\nabla w|_A^2 d\sigma dt,
\end{equation}
where $C=(\mbox{div}(a_{k\ell}B))$.

Let 
\[
D=C/2-A(B')^t.
\]
We get by putting \eqref{ani3} into \eqref{ani2}
\begin{equation}\label{ani4}
I_1=\int_Q(D\nabla w|\nabla w)dxdt+\int_\Sigma \left[(\nabla w|\nu)_A(\nabla w|B)-2^{-1}(B|\nu)|\nabla w|_A^2\right]d\sigma dt.
\end{equation}
For $I_2$, we obtain by making integrations by parts
\begin{align*}
I_2&=-\int_Q(\nabla w| \nabla (d\partial_tw))_Adxdt+\int_\Sigma (\nabla w|\nu)_A d\partial_twd\sigma dt
\\
&=- \int_Q(\nabla w|\nabla d)_A\partial_twdxdt-\int_Q d (\nabla w|\nabla \partial_tw)dxdt
\\
&\hskip 5cm +\int_\Sigma (\nabla w|\nu)_A d\partial_twd\sigma dt.
\end{align*}
As $A$ is symmetric, we have
\[
(\nabla w|\nabla \partial_tw)_A=\partial_t|\nabla w|_A^2/2.
\]
Hence
\begin{align}
I_2=-\int_Q(\nabla w&|\nabla d)_A \partial_twdxdt+\int_Q\partial_t(d/2) |\nabla w|_A^2dxdt\label{ani5}
\\
&+\int_\Sigma (\nabla w|\nu)_A d\partial_twdxdt-\int_\Omega \left[(d/2) |\nabla w|_A^2\right]_{t=t_1}^{t_2}dx.\nonumber
\end{align}
We have also
\begin{align}
I_3=\int_Q\Delta_A wbw&=-\int_Q b|\nabla w|_A^2dxdt -\int_Q w(\nabla b|\nabla w)_A dxdt \label{ani6}
\\
&\hskip 4cm +\int_\Sigma (\nabla w|\nu )_Abw d\sigma dt\nonumber
\\
&=-\int_Q b|\nabla w|_A^2dxdt +\int_Q \Delta_A(b/2) w^2 dxdt\nonumber
\\
&\hskip 1cm -\int_\Sigma (\nabla (b/2)|\nu)_Aw^2d\sigma dt +\int_\Sigma (\nabla w|\nu )_Abw d\sigma dt.\nonumber
\end{align}
Let $J_1=I_1+I_2+I_3$ and
\begin{align*}
&\mathcal{A}_1=D+\left[\partial_t(d/2)-b\right]A,
\\
&a_1=\Delta_A (b/2),
\\
&\mathcal{B}_1(w)=-(\nabla w|\nabla d)_A \partial_tw,
\\
&g_1(w)=(\nabla w|\nu)_A(\nabla w|B)-(B/2|\nu)|\nabla w|_A^2+(\nabla w|\nu)_A d\partial_tw
\\
&\hskip 2cm -(\nabla (b/2)|\nu)_Aw^2d\sigma + (\nabla w|\nu )_Abw,
\\
&h_1(w)=-\left[(d/2) |\nabla w|_A^2\right]_{t=t_1}^{t_2}.
\end{align*}
Putting together \eqref{ani4} to \eqref{ani6}, we find
\begin{align*}
J_1=\int_Q (\mathcal{A}_1\nabla w|\nabla w)dxdt &+\int_Q\mathcal{B}_1(w)dxdt+\int_Qa_1w^2dxdt.
\\
 &+\int_\Sigma g_1(w)d\sigma dt +\int_\Omega h_1(w) dx.
\end{align*}
Straightforward computations show that
\[
\mathcal{A}_1=2\tau A\nabla ^2\phi A+\tau \Upsilon_A(\phi ).
\]
Whence
\begin{align}
&J_1=\tau \int_Q \left(\left[2A\nabla ^2\phi A+\Upsilon_A(\phi )\right]\nabla w|\nabla w\right)dxdt +\int_Q\mathcal{B}_1(w)dxdt\label{ani7}
\\
&\hskip 3cm +\int_Qa_1w^2dxdt +\int_\Sigma g_1(w)d\sigma dt +\int_\Omega h_1(w) dx.\nonumber
\end{align}
We obtain, by using again an integration by parts, 
\[
I_4=\int_Q\partial_tw\partial_t(\nabla w|B)dxdt -\int_\Omega \left[\partial_tw(\nabla w|B)\right]_{t=t_1}^{t_2}dx.
\]
Hence
\[
I_4=\int_Q\partial_tw (\nabla w|\partial_tB) dxdt+\int_Q\partial_tw(\nabla \partial_t w|B) dxdt -\int_\Omega [\partial_tw(\nabla w|B)]_{t=t_1}^{t_2}dx.
\]
But
\begin{align*}
\int_Q\partial_tw(\nabla \partial_t w|B) dxdt&=-\int_Q\mbox{div}(\partial_twB)\partial_twdxdt+\int_\Sigma  (\partial_tw)^2(B|\nu)d\sigma dt
\\
&=-\int_Q\partial_tw(\nabla \partial_t w|B) dxdt-\int_Q (\partial_tw)^2\mbox{div}(B) dxdt
\\
&\hskip 4.5cm +\int_\Sigma  (\partial_tw)^2(B|\nu)d\sigma dt.
\end{align*}
Therefore
\[
2\int_Q\partial_tw(\nabla \partial_t w|B) dxdt =-\int_Q (\partial_tw)^2\mbox{div}(B) dxdt +\int_\Sigma  (\partial_tw)^2(B|\nu)d\sigma dt.
\]
In consequence
\begin{align}
I_4 =\int_Q\partial_tw &(\nabla w|\partial_tB) dxdt- \int_Q (\partial_tw)^2\mbox{div}(B/2) dxdt\label{ani8}
\\
& +\int_\Sigma  (\partial_tw)^2(B/2|\nu)d\sigma dt -\int_\Omega \left[\partial_tw(\nabla w|B)\right]_{t=t_1}^{t_2}dx.\nonumber
\end{align}
For $I_5$, we have
\begin{align}
I_5=-\int_Q (d/2)\partial_t(\partial_t w)^2  dxdt= &\int_Q\partial_t(d/2)(\partial_tw)^2  dxdt\label{ani9}
\\
&\quad -\int_\Omega \left[(d/2)(\partial_t w)^2)\right]_{t=t_1}^{t_2}dx.\nonumber
\end{align}
Also, 
\begin{align*}
I_6&=-\int_Q\partial_t^2wbwdxdt
\\
&=\int_Q\partial_tb w\partial_tw dxdt +\int_Qb (\partial_tw)^2 dxdt -\int_\Omega \left[b w\partial_tw\right]_{t=t_1}^{t_2}dx
\\
&=\int_Q\partial_t(b/2) \partial_tw^2 dxdt +\int_Qb (\partial_tw)^2 dxdt -\int_\Omega \left[b w\partial_tw\right]_{t=t_1}^{t_2}dx
\end{align*}
and hence
\begin{align}
I_6=-\int_Q\partial_t^2(b/2) w^2 &dxdt +\int_Qb (\partial_tw)^2 dxdt\label{ani10}
\\
&-\int_\Omega \left[b w\partial_tw\right]_{t=t_1}^{t_2}dx+\int_\Omega \left[(\partial_tb/2)w^2\right]_{t=t_1}^{t_2}dx.\nonumber
\end{align}
Let $J_2=I_4+I_5+I_6$ and define
\begin{align*}
&\mathfrak{a}_2=-\mbox{div}(B/2)+\partial_t(d/2)+b,
\\
&a_2=-\partial_t^2(b/2),
\\
&\mathcal{B}_2(w)=\partial_tw (\nabla w| \partial_tB),
\\
&g_2(w)=(\partial_tw)^2(B/2|\nu),
\\
&h_2(w)=- \left[\partial_tw(\nabla w|B)\right]_{t=t_1}^{t_2}- \left[d(\partial_t w)^2)\right]_{t=t_1}^{t_2}- \left[b w\partial_tw\right]_{t=t_1}^{t_2}+\left[(\partial_tb/2) w^2\right]_{t=t_1}^{t_2}.
\end{align*}
A combination of \eqref{ani8} to \eqref{ani10} gives
\begin{align*}
&J_2=\int_Q \mathfrak{a}_2(\partial_t w)^2dxdt +\int_Q\mathcal{B}_2(w)dxdt+\int_Qa_2w^2dxdt
\\
&\hskip 4.5cm +\int_\Sigma g_2(w)d\sigma dt +\int_\Omega h_2(w) dx.\nonumber
\end{align*}
Let us observe that we have by straightforward computations 
\[
\mathfrak{a}_2=2\tau \partial_t^2\phi,\quad \mathcal{B}_2=\mathcal{B}_1.
\]
Hence
\begin{align}
&J_2=2\tau \int_Q \partial_t^2\phi (\partial_t w)^2dxdt +\int_Q\mathcal{B}_1(w)dxdt+\int_Qa_2w^2dxdt\label{ani11}
\\
&\hskip 5cm +\int_\Sigma g_2(w)d\sigma dt +\int_\Omega h_2(w) dx.\nonumber
\end{align}
Let $\tilde{J}=J_1+J_2$. In light of \eqref{ani4} and \eqref{ani11} we deduce that
\begin{align*}
&\tilde{J}=\tau \int_Q \left(\left[2A\nabla ^2\phi A+\Upsilon_A(\phi )\right]\nabla w|\nabla w\right)dxdt+2\tau \int_Q \partial_t^2\phi (\partial_t w)^2dxdt 
\\
&\hskip 3cm -4\tau \int_Q\partial_tw(\nabla w|\nabla \partial_t \phi)_Adxdt+\int_Q\tilde{a}w^2dxdt
\\
&\hskip 6cm +\int_\Sigma \tilde{g}(w)d\sigma dt +\int_\Omega \tilde{h}(w) dx,
\end{align*}
where 
\[
\tilde{a}=a_1+a_2,\quad \tilde{g}=g_1+g_2,\quad \tilde{h}=h_1+h_2.
\]
As $\phi =e^{\lambda \psi}$, we have 
\begin{align*}
&\nabla ^2\phi =\lambda^2\phi (\nabla \psi\otimes \nabla \psi)+\lambda \phi \nabla ^2\psi,
\\
&
\partial_t^2\phi =\lambda^2\phi (\partial_t\psi)^2+\lambda \phi \partial_t^2\psi.
\end{align*}
This and the fact that $\nabla \partial_t\psi =0$ imply
\begin{align*}
&\left(\nabla ^2\phi A\nabla w|\nabla w\right)_A+\partial_t^2\phi (\partial_t w)^2 -2\partial_tw(\nabla w|\nabla \partial_t \phi)_A=
\\
&\hskip 2cm \lambda \phi\left[(\nabla ^2\psi A\nabla w|\nabla w)_A+\partial_t^2\psi (\partial_tw)^2\right]
\\
&\hskip 2.5cm +\lambda ^2\phi \left[(\nabla \psi|\nabla w)_A^2+(\partial_t\psi)^2(\partial_tw)^2-2\partial_t\psi \partial_tw(\nabla \psi |\nabla w)_A\right].
\end{align*}
That is we have
\begin{align*}
&\left(\nabla ^2\phi A\nabla w|\nabla w\right)_A+\partial_t^2\phi (\partial_t w)^2 -2\partial_tw(\nabla w|\nabla \partial_t \phi)_A=
\\
&\hskip 2cm \lambda \phi\left[(\nabla ^2\psi A\nabla w|\nabla w)_A+\partial_t^2\psi (\partial_tw)^2\right]
 +\lambda ^2\phi \left[(\nabla \psi |\nabla w)_A-\partial_t\psi \partial_tw\right]^2,
\end{align*}
from which we deduce, by noting that $\Upsilon_A(\phi)=\lambda\phi \Upsilon_A(\psi)$,
\begin{align}
&\tilde{J}\ge \tau \lambda \int_Q \phi\left[\left(\mathcal{A}A^{1/2}\nabla w|A^{1/2}\nabla w\right)+\partial_t^2\psi (\partial_tw)^2\right]dxdt\label{ani12}
\\
&\hskip 3.5cm +\int_Q\tilde{a}w^2dxdt+\int_\Sigma \tilde{g}(w)d\sigma dt +\int_\Omega \tilde{h}(w) dx.\nonumber
\end{align}
Here
\[
\mathcal{A}=2A^{1/2}\nabla^2\psi A^{1/2} +A^{-1/2}\Upsilon_A(\psi)) A^{-1/2}=A^{-1/2}\Theta_A(\psi)A^{-1/2}.
\]
For $I_7$, we have
\begin{align}
I_7=\int_Qaw(B|\nabla w)&dxdt=\int_Qa(B/2|\nabla w^2)dxdt \label{ani13}
\\
&=-\int_Q  \mbox{div}(aB/2)w^2dxdt
+ \int_\Sigma a(B/2|\nu)w^2d\sigma dt.\nonumber
\end{align}
Finally,
\begin{equation}\label{ani14}
I_8=\int_Q(ad/2)\partial_tw^2dxdt=-\int_Q\partial_t(ad/2)w^2dxdt +\int_\Omega \left[(ad/2)w^2\right]_{t=t_1}^{t_2}dx.
\end{equation}
Define $\hat{J}=I_7+I_8+I_9$ and 
\begin{align*}
&\hat{a}=-\mbox{div}(aB/2)-\partial_t(ad/2)+ab+\Delta_A (b/2)-\partial_t^2(b/2),
\\
&\hat{g}(w)=\tilde{g}(w)+a(B/2|\nu)w^2,
\\
&\hat{h}(w)=\tilde{h}(w)+ \left[(ad/2)w^2\right]_{t=t_1}^{t_2}.
\end{align*}
Then we have from \eqref{ani12} to \eqref{ani14}
\begin{align}
\langle &L_+w|L_-w\rangle_{L^2(Q)}=\tilde{J}+\hat{J}
\\
&\quad \ge \tau \lambda \int_Q \phi[\left(\mathcal{A}A^{1/2}\nabla w|A^{1/2}\nabla w\right)+\partial_t^2\psi (\partial_tw)^2]dxdt\label{ani15}
+\int_Q\hat{a}w^2dxdt\nonumber
\\
&\hskip 6.5cm +\int_\Sigma \hat{g}(w)d\sigma dt +\int_\Omega \hat{h}(w) dx.\nonumber
\end{align}
We prove (see details in the end of this proof) that
\begin{equation}\label{ani15.1}
\hat{a}\ge \tau^3\lambda^4\phi^3\delta ,\quad \lambda \ge \lambda_1,\; \tau \ge \tau_1.
\end{equation}
Whence
\begin{align*}
&2\langle L_+w|L_-w\rangle_{L^2(Q)}\ge 2\varkappa^{-1}\kappa \tau \lambda \int_Q \phi|\nabla w|_A^2dxdt 
\\
&\hskip 1.5cm +2\tau \lambda\int_Q\phi \partial_t^2\psi (\partial_tw)^2dxdt +2\tau^3\lambda^4\delta \int_Q\phi^3w^2dxdt+\mathscr{B}_0(w),
\end{align*}
for $\lambda \ge \lambda_1$ and $\tau \ge \tau_1$, with 
\[
\mathscr{B}_0(w)=\int_\Sigma \hat{g}(w)d\sigma dt +\int_\Omega \hat{h}(w) dx.
\]
On the other hand, we find by making twice integration by parts
\begin{align*}
\int_Q(L_+w)\phi wdxdt=-\int_Q\phi |\nabla w|_A^2&+\int_Q\phi (\partial_tw)^2dxdt
\\
&+\int_Q(a-b/2)w^2dxdt+\mathscr{B}_1(w),
\end{align*}
with
\begin{align*}
\mathscr{B}_1(w)=\int_\Sigma &\phi (\nabla w|\nu)_Awd\sigma dt-\frac{1}{2}\int_\Sigma (\nabla \phi|\nu)_Aw^2d\sigma dt
\\
&\quad + \int_\Omega -[\phi w\partial_tw]_{t=t_1}^{t_2}dx+\frac{1}{2}\int_\Omega \left[w^2\partial_t\phi\right]_{t=t_1}^{t_2}dx.
\end{align*}
Let $\epsilon >0$ to be determined later. Then Cauchy-Schwarz's inequality yields
\begin{align*}
\int_Q(L_+w)^2dxdt&+\epsilon ^2\tau^2\lambda ^2\int_Q\phi ^2w^2dxdt\ge 
\\
&-\frac{\epsilon \tau \lambda }{2}\int_Q\phi |\nabla w|_A^2dxdt +\frac{\epsilon \tau \lambda }{2}\int_Q\phi (\partial_tw)^2dxdt
\\
&\hskip 1.5cm +\frac{\epsilon \tau \lambda }{2}\int_Q(a-b/2)w^2dxdt+\frac{\epsilon \tau \lambda }{2}\mathscr{B}_1(w).
\end{align*}
Using that $a=\tau^2\lambda^2\phi^2\left(|\nabla \psi|_A^2-(\partial_t\psi)^2\right)$ we get
\begin{align*}
\int_Q&(L_+w)^2dxdt\ge 
\\
&\hskip 1cm -\frac{\epsilon \tau \lambda }{2}\int_Q\phi |\nabla w|_A^2dxdt +\frac{\epsilon \tau \lambda }{2}\int_Q\phi (\partial_tw)^2dxdt
\\
&\hskip 2cm +\frac{\epsilon \tau^3 \lambda ^3}{2}\int_Q\phi^3\left(|\nabla \psi|_A^2-(\partial_t\psi)^2\right)w^2dxdt-\epsilon ^2\tau^2\lambda ^3\int_Q\phi ^2w^2dxdt
\\
&\hskip 6.5cm -\frac{\epsilon \tau \lambda }{2}\int_Q(b/2)w^2dxdt+\frac{\epsilon \tau \lambda}{2}\mathscr{B}_1(w).
\end{align*}
Hence
\begin{align*}
&2\langle L_+w|L_-w\rangle_{L^2(Q)}+\int_Q(L_+w)^2dxdt \ge (2\varkappa^{-1}\kappa -\epsilon/2) \tau \lambda \int_Q \phi|\nabla w|_A^2dxdt 
\\
&+\tau \lambda\int_Q\phi (\epsilon/2+ 2\partial_t^2\psi) (\partial_tw)^2dxdt+2\tau^3\lambda^4\delta \int_Q\phi^3w^2dxdt
\\
&\hskip 1cm +\frac{\epsilon \tau^3 \lambda ^3}{2}\int_Q\phi^3(|\nabla \psi|_A^2-(\partial_t\psi)^2)w^2dxdt
\\
&\hskip 1.5cm -\epsilon ^2\tau^2\lambda ^2\int_Q\phi ^2w^2dxdt-\frac{\epsilon \tau \lambda ^2}{2}\int_Q(b/2)w^2dxdt+\mathscr{B}(w),\quad \lambda \ge \lambda_1,\; \tau \ge \tau_1.
\end{align*}
Here $\mathscr{B}(w)=\mathscr{B}_0(w)+\frac{\epsilon \tau \lambda }{2}\mathscr{B}_1(w)$.

We take $\epsilon=2\varkappa^{-1}\kappa$ in the preceding inequality and we use inequality \eqref{int2}. We obtain, by noting that in the right hand side of the last inequality, the fourth, fifth and sixth terms can be absorbed by the third term, 
\begin{align*}
& 2\langle L_+w|L_-w\rangle_{L^2(Q)}+\int_Q(L_+w)^2dxdt \ge \varkappa^{-1}\kappa \tau \lambda \int_Q \phi[|\nabla w|_A^2+(\partial_tw)^2]dxdt
\\
&\hskip 4.5cm +\tau^3\lambda^4\delta\int_Q\phi^3w^2dxdt+\mathscr{B}(w) ,\quad \lambda \ge \lambda_2,\; \tau \ge \tau_2.
\end{align*}
We find by making elementary calculations
\[
\aleph_3|\mathscr{B}(w)| \le \int_{\partial Q} e^{2\tau \phi}\left[\tau^3\lambda ^3\phi^3u^2+\tau \lambda \phi \left(|\nabla_A u|^2+(\partial_tu)^2\right)\right]d\mu,
\]
for $\lambda\ge \lambda_3$ and $\mu \ge \mu_3$. This and 
\[
\|Lw\|_{L^2(Q)}^2\ge 2\langle L_+w|L_-w\rangle_{L^2(Q)}+\|L_+w\|_{L^2(Q)}^2
\]
imply
\begin{align}
&\aleph\int_Q\left[\tau^3\lambda ^4\phi^3 w^2+\tau \lambda  \phi |\mathbf{D}_A w|^2\right]dxdt \label{C1.1}
\\
&\hskip1cm \le \int_Q\left(\mathcal{L}_{A,0}^ww\right)^2dxdt +\int_{\partial Q} \left[\tau^3\lambda ^3\phi^3w^2+\tau \lambda \phi |\mathbf{D}_A w|^2\right]d\mu .\nonumber
\end{align}
We take in this inequality $w=\Phi^{-1}u$ with $u\in H^2 (Q;\mathbb{R})$. In light of the identities
\begin{align*}
&\Phi^{-1}\nabla u= -\tau\lambda w\nabla \psi +\nabla w,
\\
& \Phi^{-1}\partial_t u= -\tau\lambda w\partial_t \psi +\partial_t w,
\end{align*}
we obtain an inequality similar to \eqref{C1.0} which leads to \eqref{C1.0} by observing that the additional terms in the right hand side appearing in this intermediate inequality can be absorbed by the terms in left hand side.

\textit{Proof of \eqref{ani15.1}.} Set
\[
\chi =|\nabla \psi|_A^2-(\partial_t\psi)^2.
\]
We have
\begin{align*}
&\Delta_A\phi=\mbox{div}(A\nabla e^{\lambda \psi})=\mbox{div}(\lambda \phi A\nabla \psi)=\lambda^2\phi |\nabla \psi|_A^2+\lambda \phi\Delta_A\psi,
\\
&\partial_t^2\phi=\partial_t(\lambda \phi\partial_t\psi)=\lambda^2\phi(\partial_t\psi)^2+\lambda \phi \partial_t^2\psi.
\end{align*}
That is 
\begin{equation}\label{tc1}
\mathcal{L}_{A,0}^w\phi= \lambda^2\phi\chi +\lambda\phi \mathcal{L}_{A,0}^w\psi.
\end{equation}
In light of \eqref{tc1} we get
\begin{align*}
&a(x,t)=\tau^2\left(|\nabla \phi|_A^2-(\partial_t\phi)^2\right)=\tau^2\lambda^2\phi^2\chi,
\\
&b(x,t)=-\tau \left(\Delta_A  \phi -\partial_t^2\phi\right)=-\tau\lambda^2\phi\chi -\tau\lambda\phi \mathcal{L}_{A,0}^w\psi,
\\
&B=-2\tau  A\nabla\phi=-2\tau \lambda \phi A\nabla \psi ,
\\
&d=2\tau\partial_t\phi= 2\tau \lambda \phi\partial_t\psi .
\end{align*}
Since 
\[
-aB/2= \tau^3\lambda^3\phi^3\chi A\nabla \psi,
\]
we find
\begin{equation}\label{tc2}
-\mbox{div}(aB/2)=3\tau^3\lambda^4\phi^3\chi |\nabla \psi |^2+\tau^3\lambda^3\phi^3\mbox{div}(\chi A\nabla \psi).
\end{equation}
Also, as
\[
-ad/2=-\tau^3\lambda^3\phi^3\chi \partial_t\psi,
\]
we obtain
\begin{equation}\label{tc3}
-\partial_t(ad/2)=-3\tau^3\lambda^4\phi^3\chi (\partial_t\psi)^2- \tau^3\lambda^3\phi^3\partial_t(\chi \partial_t\psi).
\end{equation}
We get by putting together \eqref{tc2} and \eqref{tc3} 
\begin{equation}\label{tc4}
-\mbox{div}(aB/2)-\partial_t(ad/2)= 3\tau^3\lambda^4\phi^3\chi^2+\tau^3\lambda^3\phi^3\left[\mbox{div}(\chi A\nabla \psi)-\partial_t(\chi \partial_t\psi)\right].
\end{equation}
 On the other hand,
 \begin{equation}\label{tc5}
 ab=-\tau^3\lambda^4\phi^3\chi^2-\tau^3\lambda ^3\phi^3\chi\mathcal{L}_{A,0}^w\psi .
 \end{equation}
Set
\[
\chi_1= \mbox{div}(\chi A\nabla \psi)-\partial_t(\chi \partial_t\psi)-\chi\mathcal{L}_{A,0}^w\psi=(\nabla \chi |\nabla\psi)_A-\partial_t\chi\partial_t\psi.
\]
A combination of \eqref{tc4} and \eqref{tc5} yields
 \begin{equation}\label{tc6}
 -\mbox{div}(aB/2)-\partial_t(ad/2)+ab=2\tau^3\lambda^4\phi^3\chi^2+\tau^3\lambda^3\phi^3\chi_1 .
 \end{equation}
We have again  from \eqref{tc1}
\begin{align*}
\mathcal{L}_{A,0}^wb&=\mathcal{L}_{A,0}^w\left( -\tau\lambda^2\phi\chi -\tau\lambda\phi \mathcal{L}_{A,0}^w\psi \right)
\\
&=- \left(\lambda^2\phi\chi +\lambda\phi \mathcal{L}_{A,0}^w\psi\right)\left( \tau\lambda^2\chi +\tau\lambda\mathcal{L}_{A,0}^w\psi \right)
\\
&\hskip 2cm -\phi \left( \tau\lambda^2\mathcal{L}_{A,0}^w\chi +\tau\lambda\left(\mathcal{L}_{A,0}^w\right)^2\psi \right).
\end{align*}
Therefore 
\begin{equation}\label{tc7}
\mathcal{L}_{A,0}^w(b/2)\ge -\tau \lambda ^4\phi \delta ,\quad \lambda \ge \lambda_4,\; \tau >0,
\end{equation}
where we used that $\chi ^2\ge \delta$.

Inequality \eqref{ani15.1} then follows by combining \eqref{tc6} and \eqref{tc7} and using that
\[
\hat{a}= -\mbox{div}(aB/2)-\partial_t(ad/2)+ab+\mathcal{L}_{A,0}^w(b/2).
\]
\end{proof}

Define $\partial_{\nu_A}\psi_0$ by
\[
\partial_{\nu_A}\psi_0=(\nabla\psi_0|\nu)_A 
\]
and set
\[ 
\Gamma_+=\Gamma_+^{\psi_0}=\left\{x\in \Gamma ;\; \partial_{\nu_A}\psi_0(x)>0\right\},\quad \Sigma_+=\Sigma_+^{\psi_0}=\Gamma_+\times (t_1,t_2).
\]

Let $w\in H^2(Q,\mathbb{R})$ satisfying $w=0$ on $\Sigma$. In that case it is straightforward to check that $\hat{g}(w)$ defined in the preceding proof takes the form
\begin{equation}\label{id1}
\hat{g}(w)=-\tau \lambda \phi (\partial_\nu w)^2|\nu|_A^2\partial_{\nu_A}\psi_0.
\end{equation}
Furthermore, if $u=\Phi w$ then 
\begin{equation}\label{id2}\partial_\nu u=\Phi \partial_\nu w. 
\end{equation}
In light of  identities \eqref{id1} and \eqref{id2}, slight modifications of the last part of the preceding proof enable us to establish the following result.

\begin{theorem}\label{CarlemanTheorem1.2}
There exist three constants $\aleph=\aleph(\mathfrak{d})$, $\lambda^\ast=\lambda^\ast(\mathfrak{d})$ and $\tau^\ast =\tau^\ast(\mathfrak{d})$ so that, for any $\lambda \ge \lambda^\ast$, $\tau \ge \tau^\ast$ and $u\in H^2(Q,\mathbb{R})$ satisfying $u=0$ on $\Sigma$ and $u=\partial_tu=0$ in $\Omega \times\{t_1,t_2\}$, we have
\begin{align}
&\aleph\int_Qe^{2\tau \phi}\left[\tau^3\lambda ^4\phi^3 u^2+\tau \lambda  \phi |\mathbf{D}_A u|^2\right]dxdt \label{C1.1}
\\
&\hskip1cm \le \int_Qe^{2\tau \phi}\left[\mathcal{L}_{A,0}^wu\right]^2dxdt +\tau \lambda \int_{\Sigma_+} e^{2\tau \phi}\phi (\partial_\nu u)^2 d\sigma dt.\nonumber
\end{align}
\end{theorem}

Let us see that Theorem \ref{CarlemanTheorem1} remains valid whenever  we add to $\mathcal{L}_{A,0}^w$ a first order operator. Consider then the operator
\[
\mathcal{L}_A^w=\mathcal{L}_{A,0}^w+q_0\partial_t+\sum_{j=1}^n q_j\partial_j+p,
\]
where $q_0,\ldots ,q_n$ and $p$ belong to $L^\infty (Q,\mathbb{C})$ and satisfy
\[
\max_{0\le i\le n}\|q_i\|_{L^\infty(Q)}\le \mathfrak{m},\quad \|p\|_{L^\infty(Q)} \le \mathfrak{m}.
\]
Let $u=v+iw\in H^2(Q,\mathbb{C})$ and apply Theorem \ref{CarlemanTheorem1} to both $v$ and $w$. We obtain by adding side by side the inequalities we obtain by taking in \eqref{C1.0} $u=v$ and then $u=w$
\begin{align}
&\aleph\int_Qe^{2\tau \phi}\left[\tau^3\lambda ^4\phi^3 |u|^2+\tau \lambda  \phi |\mathbf{D}_A u|^2\right]dxdt \label{C28}
\\
&\hskip1cm \le \int_Qe^{2\tau \phi}|\mathcal{L}_{A,0}^wu|^2dxdt +\int_{\partial Q} e^{2\tau \phi}\left[\tau^3\lambda ^3\phi^3|u|^2+\tau \lambda \phi |\mathbf{D}_A u|^2\right]d\mu .\nonumber
\end{align}
Since 
\[
|\mathcal{L}_{A,0}^wu|^2\le 2|\mathcal{L}_A^wu|^2+2(n+2)\mathfrak{m}^2(\varkappa |\nabla u|_A^2+|\partial_tu|^2+|u|^2)
\]
and the term
\[
2(n+2)\mathfrak{m}^2\int_Qe^{2\tau \phi}(\varkappa |\nabla u|_A^2+|\partial_tu|^2+|u|^2)dxdt
\]
can be absorbed by the left hand side of \eqref{C28}, we get the following result:

\begin{corollary}\label{CarlemanCorollary1}
We find three constants $\aleph=\aleph(\mathfrak{d})$, $\lambda^\ast=\lambda^\ast(\mathfrak{d})$ and $\tau^\ast =\tau^\ast(\mathfrak{d})$ so that, for any $\lambda \ge \lambda^\ast$, $\tau \ge \tau^\ast$ and $u\in H^2(Q,\mathbb{C})$, we have
\begin{align}
&\aleph\int_Qe^{2\tau \phi}\left[\tau^3\lambda ^4\phi^3 |u|^2+\tau \lambda  \phi |\mathbf{D}_A u|^2\right]dxdt \label{C29}
\\
&\hskip1cm \le \int_Qe^{2\tau \phi}|\mathcal{L}_A^wu|^2dxdt+\int_{\partial Q} e^{2\tau \phi}\left[\tau^3\lambda ^3\phi^3|u|^2+\tau \lambda \phi |\mathbf{D}_A u|^2\right]d\mu ,\nonumber
\end{align}
\end{corollary}

Finally, we note that the preceding arguments allow us also  to prove the following corollary.

\begin{corollary}\label{CarlemanCorollary2}
There exist three constants $\aleph=\aleph(\mathfrak{d})$, $\lambda^\ast=\lambda^\ast(\mathfrak{d})$ and $\tau^\ast =\tau^\ast(\mathfrak{d})$ so that, for any $\lambda \ge \lambda^\ast$, $\tau \ge \tau^\ast$ and $u\in H^2(Q,\mathbb{C})$ satisfying $u=0$ on $\Sigma$ and $u=\partial_tu=0$ in $\Omega \times\{t_1,t_2\}$, we have
\begin{align}
&\aleph\int_Qe^{2\tau \phi}\left[\tau^3\lambda ^4\phi^3 |u|^2+\tau \lambda  \phi |\mathbf{D}_A u|^2\right]dxdt \label{C1.1}
\\
&\hskip1cm \le \int_Qe^{2\tau \phi}\left|\mathcal{L}_A^wu\right|^2dxdt +\tau \lambda \int_{\Sigma_+} e^{2\tau \phi}\phi |\partial_\nu u|^2 d\sigma dt,\nonumber
\end{align}
\end{corollary}

We close this subsection by a one-parameter Carleman inequality that we obtain as a special case of Corollary \ref{CarlemanCorollary1} in which we fixed $\lambda \ge \lambda^\ast$.

\begin{theorem}\label{CarlemanTheoremWaveLocal}
There exist two constants $\aleph=\aleph(\mathfrak{d})>0$ and $\tau^\ast =\tau^\ast(\mathfrak{d})>0$ so that, for any $\tau \ge \tau^\ast$ and $u\in C_0^\infty(Q,\mathbb{C})$, we have
\[
\sum_{|\alpha|\le 1}\tau^{2(2-|\alpha|)}\int_Qe^{2\tau \phi}|\partial ^\alpha u|^2dxdt  \le \aleph \tau \int_Qe^{2\tau \phi}|\mathcal{L}_A^wu|^2dxdt,
\]
\end{theorem}

\subsection{Geometric form of the Carleman inequality}

Pick $A\in \mathscr{M}(\Omega ,\varkappa ,\mathfrak{m})$ and let $(a^{k\ell}(x))=(a_{k\ell}(x))^{-1}$, $x\in \overline{\Omega}$. Consider then on $\overline{\Omega}$ the Riemannian metric $g$ defined as follows
\[
g_{k\ell}(x)=|\mbox{det}(A)|^{1/(n-2)}a^{k\ell}(x),\quad x\in \overline{\Omega}.
\]
Set then
\[
(g^{k\ell}(x))=(g_{k\ell}(x))^{-1},\quad |g(x)|=|\mbox{det}((g_{k\ell}(x))|,\quad x\in \overline{\Omega}.
\]
As usual  define on $T_x\overline{\Omega}=\mathbb{R}^n$ the inner product
\[
(X|Y)_{g(x)}=\sum_{k,\ell =1}g_{k\ell}(x)X_kY_k,\quad X=\sum_{k=1}^nX_i\partial_i\in \mathbb{R}^n,\; Y=\sum_{k=1}^nY_i\partial_i\in \mathbb{R}^n,
\]
where $(\partial_1,\ldots ,\partial_n)$ is the dual basis of the Euclidean basis of $\mathbb{R}^n$. Set
\[
|X|_{g(x)}=(X|X)_{g(x)}^{1/2},\quad X=\sum_{k=1}^nX_i\partial_i\in \mathbb{R}^n.
\]
We use for notational convenience $(X|Y)_g$ and $|X|_g$ instead of $(X|Y)_{g(x)}$ and $|X|_{g(x)}$.

Recall that the gradient of $u\in H^1(\Omega )$ is  the vector field given by 
\[
\nabla_gu(x)=\sum_{k,\ell=1}^n g^{k\ell}(x)\partial_ku(x)\partial_\ell,\quad x\in \Omega,
\]
and the divergence of a vector field $X=\sum_{\ell=1}^nX_\ell\partial_k$ with $X_\ell \in H^1(\Omega)$, $1\le \ell \le n$, is defined as follows
\[
\mbox{div}_g(X)(x)=\frac{1}{\sqrt{|g(x)|}}\sum_{\ell=1}\partial_\ell (\sqrt{|g(x)}|X_\ell(x)),\quad x\in \Omega.
\]
The usual Laplace-Betrami operator associated to the metric $g$ is given,  for $u\in H^2(\Omega )$, by
\[
\Delta_gu(x)=\mbox{div}_g\nabla_gu(x)=\frac{1}{\sqrt{|g(x)|}}\sum_{k,\ell=1}\partial_\ell \left(\sqrt{|g(x)|}g^{k\ell}(x)\partial_ku(x)\right),\quad x\in \Omega.
\]
Straightforward computations show that $\Delta_Au=\sqrt{|g|}\Delta_gu$, from which we deduce the following identity
\begin{equation}\label{gf}
\Delta_gu =\Delta_{\sqrt{|g|}^{-1}A}u-2\left(\nabla\sqrt{|g|}^{-1}\Big|\nabla u\right)_A-u\Delta_A\sqrt{|g|}^{-1}.
\end{equation}
We assume in this subsection that  $\psi$ is a weight function for the wave operator $\Delta_{\sqrt{|g|}^{-1}A}-\partial_t^2$ with $\sqrt{|g|}^{-1}A$-pseudo convexity constant $\kappa>0$. Hereafter
\[
\mathcal{L}_g^wu=\Delta_gu-\partial_t^2u+(P|\nabla_g u)_g+q_1\partial_tu+q_0u,\quad u\in H^2(\Omega ),
\]
where $P=\sum_{\ell=1}^nP_i\partial_i$, $q_0$ and $q$ are so that $P_1,\ldots ,P_n, q_0,q_1$ belong to $L^\infty (Q;\mathbb{C})$ and satisfy
\[
\|P_\ell \|_{L^\infty(Q)}\le \mathfrak{m},\; 1\le \ell\le n,\quad \|q_0\|_{L^\infty(Q)}\le \mathfrak{m},\quad \|q_1\|_{L^\infty(Q)}\le \mathfrak{m}.
\]
As in the preceding section $\mathfrak{d}=(\Omega , t_1,t_2,\mathfrak{m}, \varkappa,\kappa ,\delta,\mathfrak{b})$ with $\mathfrak{b}\ge \|\psi\|_{C^4(\overline{Q})}$. In light of \eqref{gf} we deduce the following Carleman inequality from Corollary \ref{CarlemanCorollary1} in which
\[
|\mathbf{D}_gu|_g^2=|\nabla_gu|_g^2+|\partial_tu|^2,
\]
$dV_{\overline{\Omega}} =\sqrt{|g|}dx_1\ldots dx_n$ is the volume form and
\[
d\mu =\mathds{1}_{\Gamma \times (t_1,t_2)}dV_\Gamma dt +\mathds{1}_{\Omega \times \{t_1,t_2\}}dV_{\overline{\Omega}}\delta_t,
\]
where $dV_\Gamma$ is the volume form on $\Gamma$.

\begin{theorem}\label{CarlemanTheoremGeo1}
There exist three constants $\aleph=\aleph(\mathfrak{d})$, $\lambda^\ast=\lambda^\ast(\mathfrak{d})$ and $\tau^\ast =\tau^\ast(\mathfrak{d})$ so that, for any $\lambda \ge \lambda^\ast$, $\tau \ge \tau^\ast$ and $u\in H^2(Q,\mathbb{C})$, we have
\begin{align}
&\aleph\int_Qe^{2\tau \phi}\left[\tau^3\lambda ^4\phi^3 |u|^2+\tau \lambda  \phi |\mathbf{D}_g u|_g^2\right]dV_{\overline{\Omega}}dt \label{C1}
\\
&\hskip1cm \le \int_Qe^{2\tau \phi}\left|\mathcal{L}_g^wu\right|^2dxdt+\int_{\partial Q} e^{2\tau \phi}\left[\tau^3\lambda ^3\phi^3|u|^2+\tau \lambda \phi |\mathbf{D}_g u|_g^2\right]d\mu ,\nonumber
\end{align}
\end{theorem}

\subsection{Unique continuation}

We assume, for simplicity, in the present subsection that $t_1=-\mathfrak{t}$ and $t_2=\mathfrak{t}$ with $\mathfrak{t}>0$.

We start with unique continuation across a particular convex hypersurface. To this end, we set
\begin{align*}
E_+(\tilde{x},c)=\{x=(x',x_n)\in \mathbb{R}^n ;\; 0\le x_n-&\tilde{x}_n<c
\\
&\mbox{and}\; x_n-\tilde{x}_n\ge |x'-\tilde{x}'|^2/c\},
\end{align*}
where $\tilde{x}=(\tilde{x}',\tilde{x}_n)\in \mathbb{R}^n$ and $c>0$.

\begin{theorem}\label{theoremUCw1}
Suppose, for some $r>0$, that $B(\tilde{x},r)\Subset \Omega$. We find $0<c^\ast=c^\ast (\varkappa ,\mathfrak{m})$ with the property that, for any $0<c<c^\ast$, there exist $\tilde{\mathfrak{t}}=\tilde{\mathfrak{t}}(c,\varkappa)$ and $0<\rho=\rho(c,\varkappa)<r$  so that, if $\mathfrak{t}\ge\tilde{\mathfrak{t}}$ then there exists  $0<\mathfrak{t}_0=\mathfrak{t}_0(\varkappa ,\mathfrak{m},\mathfrak{t})\le \mathfrak{t}$ satisfying:  for any $u\in H^2(\mathcal{Q};\mathbb{C})$, with $\mathcal{Q}=B(\tilde{x},r)\times(-\mathfrak{t},\mathfrak{t})$, such that $\mathcal{L}_A^wu=$ in $\mathcal{Q}$ and $\mbox{supp}(u(\cdot ,t))\cap  B(\tilde{x},r)\subset E_+(\tilde{x},c)$, $t\in (-\mathfrak{t},\mathfrak{t})$, we have $u=0$ in $B(\tilde{x},\rho)\times (-\mathfrak{t}_0,\mathfrak{t}_0)$.
\end{theorem}

\begin{proof}
Let $\tilde{x}\in \Omega$ and $\mathfrak{e}_n=(0,1)\in \mathbb{R}^{n-1}\times \mathbb{R}$. Set $x_0=\tilde{x}+c\mathfrak{e}_n$ and let $0<r_0\le\min(r,c/2)$. Define
\[
\psi_0(x)=\psi_0(x',x_n)=|x-x_0|^2/2.
\]
As 
\[
|x-x_0|\le |x-\tilde{x}|+|\tilde{x}-x_0|\le r_0+c\le 3c/2,\quad x\in B(\tilde{x},r_0),
\]
we find in a straightforward manner that, for some constant $\aleph =\aleph (\mathfrak{m})$, we have
\[
(\Theta(\psi_0)(x)\xi |\xi )\ge 2\varkappa ^2-\aleph c,\quad x\in B(\tilde{x},r_0),\; \xi \in \mathbb{R}^n,
\]
We fix $0<c<c^\ast=\varkappa^2/\aleph$. With this choice of $c$ we get
\[
(\Theta(\psi_0)(x)\xi |\xi )\ge \varkappa ^2|\xi|^2,\quad x\in B(\tilde{x},r_0),\; \xi \in \mathbb{R}^n.
\]
It is then not difficult to check that, where $E_+=E_+(\tilde{x},c)$,
\[
(E_+\cap B(\tilde{x},r_0)) \setminus\{\tilde{x}\}\subset \left\{ x\in B(\tilde{x},r_0)\setminus\{\tilde{x}\};\; \psi_0 (x)<\psi_0(\tilde{x})=c^2/2\right\}.
\]
Pick $\chi \in C_0^\infty (B(\tilde{x},r_0))$ satisfying $\chi =1$ in $B(\tilde{x},\rho_1)$, for some fixed $ 0<\rho_1<r_0$. Fix then $\epsilon >0$ in such a way that
\[
E_+\cap [B(\tilde{x},r)\setminus \overline{B}(\tilde{x},\rho_1 )]\subset \left\{ x\in B(\tilde{x},r);\; \psi_0 (x)<\psi_0 (\tilde{x})-\epsilon \right\}.
\]
Also, choose $0<\rho_0<\rho_1$ such that
\[
E_+\cap B(\tilde{x},\rho_0)\subset \left\{ x\in B(\tilde{x},r);\; \psi_0 (x)>\psi_0 (\tilde{x})-\epsilon/2 \right\}.
\]
We are going to apply Corollary \ref{CarlemanCorollary1} with $Q$ substituted by $\mathcal{Q}=B(\tilde{x},r)\times (-\mathfrak{t},\mathfrak{t})$. Let
\[
\psi (x,t)=\psi_0(x)-\gamma t^2/2 +C,
\]
where the constant $C >0$ is chosen sufficiently large in order to guarantee that $\psi \ge 0$. From Example \ref{example1} we can easily see that $\phi=e^{\lambda \psi}$ is a weight function for $\mathcal{L}_{A,0}^w$ in $\mathcal{Q}$ provided that $0<\gamma <\min (c/(2\mathfrak{t}),\varkappa/4)$.

Suppose that $\mathfrak{t}$ is chosen sufficiently large in such a way $c/(2\mathfrak{t})<\varkappa/4$. In that case we can take $\gamma =c/(4\mathfrak{t})$.

Fix $\gamma$ as above and let $0<\varrho\le1$ to be determined later. Let $\lambda ^\ast$ and $\tau^\ast$ be as in Corollary \ref{CarlemanCorollary1}. In the sequel we fix $\lambda \ge \lambda ^\ast$. Set
\begin{align*}
&\mathbf{Q}_0=[B(\tilde{x},\rho_0)\cap E_+]\times (-\varrho \mathfrak{t},\varrho \mathfrak{t}),
\\
&\mathbf{Q}_1= \left\{E_+\cap\left[B(\tilde{x},r_0)\setminus \overline{B}(\tilde{x},\rho_1)\right]\right\}\times (-\mathfrak{t},\mathfrak{t}),
\\
&\mathbf{Q}_2=[B(\tilde{x},r_0)\cap E_+]\times [(-\mathfrak{t},- \mathfrak{t}/2)\cup (\mathfrak{t}/2,\mathfrak{t})].
\end{align*}
Then straightforward computations show
\begin{align*}
&e^{\lambda \psi}\ge c_0=e^{\lambda (c^2/2-\epsilon/2-\gamma \varrho^2 \mathfrak{t}^2/2+\delta)}\quad \mbox{in}\; \mathbf{Q}_0,
\\
&e^{\lambda \psi}\le c_1=e^{\lambda (c^2/2-\epsilon+\delta )}\quad \mbox{in}\; \mathbf{Q}_1,
\\
&e^{\lambda \psi}\le c_2=e^{\lambda (c^2/2-\gamma \mathfrak{t}^2/8+\delta )}\quad \mbox{in}\; \mathbf{Q}_2.
\end{align*}
In these inequalities we substitute $\gamma \mathfrak{t}$ by $c/4$ in order to get
\begin{align*}
&e^{\lambda \psi}\ge c_0=e^{\lambda (c^2/2-\epsilon/2- c \varrho^2 \mathfrak{t}/8+\delta)}\quad \mbox{in}\; \mathbf{Q}_0,
\\
&e^{\lambda \psi}\le c_1=e^{\lambda (c^2/2-\epsilon+\delta )}\quad \mbox{in}\; \mathbf{Q}_1,
\\
&e^{\lambda \psi}\le c_2=e^{\lambda (c^2/2-c \mathfrak{t}/32+\delta )}\quad \mbox{in}\; \mathbf{Q}_2.
\end{align*}

If $\mathfrak{t}\ge 2\epsilon/c$ we choose $\varrho$ so that $c\varrho^2\mathfrak{t}=2\epsilon$. In that case we have
\begin{align*}
&e^{\lambda \psi}\ge c_0=e^{\lambda (c^2/2-3\epsilon/4+\delta)}\quad \mbox{in}\; \mathbf{Q}_0,
\\
&e^{\lambda \psi}\le c_1=e^{\lambda (c^2/2-\epsilon+\delta )}\quad \mbox{in}\; \mathbf{Q}_1,
\\
&e^{\lambda \psi}\le c_2=e^{\lambda (c^2/2-c \mathfrak{t}/32+\delta )}\quad \mbox{in}\; \mathbf{Q}_2.
\end{align*}
In consequence $c_1<c_0$. Furthermore, if $\mathfrak{t}> 24\epsilon/c$ then we have also $c_2<c_0$.

Let $u\in H^2(\mathcal{Q};\mathbb{C})$ satisfying $\mathcal{L}_A^wu=0$ in $\mathcal{Q}$ and $\mbox{supp}(u(\cdot ,t))\cap  B(\tilde{x},r)\subset E_+$, $t\in (-\mathfrak{t},\mathfrak{t})$. Define $v=\chi (x)\vartheta (t)u$ with $\vartheta \in C_0^\infty ((-\mathfrak{t},\mathfrak{t}))$ satisfying $\vartheta =1$ in $[- \mathfrak{t}/2,\mathfrak{t}/2]$.

As $\mathcal{L}_A^wu=0$ in $\mathcal{Q}$, elementary computations show that $\mathcal{L}_A^wv=f_1+f_2$ in $\mathcal{Q}$, where
\begin{align*}
&f_1=2\vartheta (\nabla \chi |\nabla u)_A+\vartheta u\mbox{div}(A\nabla\chi) +\vartheta u\sum_{j=1}^n q_j\partial_j\chi ,
\\
&f_2= 2\chi \partial_tu\partial_t\vartheta +\chi u\partial_t^2\vartheta +q_0\chi u\partial_t\vartheta .
\end{align*}
Taking into account that $\mbox{supp}(f_1)\subset \mathbf{Q}_1$ and $\mbox{supp}(f_2)\subset \mathbf{Q}_2$, we find by applying Corollary \ref{CarlemanCorollary1}, where $\tau \ge \tau^\ast$, 
\begin{align}
&\int_{B(\tilde{x},\rho_0)\times (-\varrho \mathfrak{t},\varrho \mathfrak{t})}|u|^2dxdt \label{ucw1}
\\
&\hskip 1cm =\int_{\mathbf{Q}_0}|v|^2dxdt \nonumber
\\
&\hskip 1cm \le \aleph \tau^{-3}\left[ e^{-\tau (c_0 -c_1)}\int_{\mathbf{Q}_1}|f_1|^2dxdt+e^{-\tau (c_0 -c_2)}\int_{\mathbf{Q}_2}|f_2|^2dxdt\right] .\nonumber
\end{align}
Passing then to the limit, as $\tau$ tends to $\infty$, in \eqref{ucw1} in order to obtain that $u=0$ in $B(\tilde{x},\rho_0)\times (-\varrho \mathfrak{t},\varrho \mathfrak{t})$. 
\end{proof}

\begin{definition}\label{gpucw}
We will say that $\mathcal{L}_A^w$ has the weak unique continuation property in $Q$ if we can find $0<\tau \le \mathfrak{t}$ so that, for any open subset $\mathcal{O}$  with $\overline{\mathcal{O}}\subsetneqq \Omega$, there exists an open subset $\mathcal{O}_0$ with $\mathcal{O}_0\supsetneqq \overline{\mathcal{O}}$ so that, for any $u\in H^2(Q;\mathbb{C})$ satisfying $\mathcal{L}_A^wu=0$ in $Q$ and $u=0$  in $\mathcal{O}\times (-\mathfrak{t},\mathfrak{t})$, we have $u=0$ in $\mathcal{O}_0\times (-\tau,\tau)$.
\end{definition}

\begin{theorem}\label{theoremUCw2}
There exist a universal constant $\mathfrak{t}^\ast>0$ and a neighborhood $\mathcal{N}$ of $\mathbf{I}$ in $C^{2,1}(\overline{\Omega};\mathbb{R}^n\times\mathbb{R}^n)$ so that, for each $\mathfrak{t}\ge \mathfrak{t}^\ast$ and $A\in \mathcal{N}$, $\mathcal{L}_A^w$ has the weak unique continuation property in $Q$.
\end{theorem}

\begin{proof}
Set 
\[
H=\left\{ (x',x_n)\in \mathbb{R}^{n-1}\times \mathbb{R};\; x_n=0\right\}.
\]
Let $\mathcal{N}_0$ be the neighborhood of $\mathbf{I}$ in $C^{2,1}(\overline{\Omega};\mathbb{R}^n\times\mathbb{R}^n)$ given by Lemma \ref{lemmaPCH}. Pick $A\in \mathcal{N}_0$ and let $u\in H^2(Q;\mathbb{C})$ satisfying $\mathcal{L}_A^w=0$ in $Q$ and $u=0$  in $\mathcal{O}\times (-\mathfrak{t},\mathfrak{t})$, where $\mathcal{O}$ is an open subset satisfying $\overline{\mathcal{O}}\subsetneqq \Omega$.

Fix $y\in \partial \mathcal{O}\cap \Omega$ and $r>0$ so that $B(y,r)\subset \Omega$. Pick then  $y_0\in \mathcal{O}\cap B(y,r)$ sufficiently close to $y$ in such a way that $\partial B(y_0, d)\cap \partial \mathcal{O}\ne \emptyset$, with $d=\mbox{dist}(y_0,\partial \mathcal{O})$. Let then $z\in \partial B(y_0, d)\cap \partial \mathcal{O}$. Making a translation and a change of coordinates if necessary, we may assume that $z=0$ and $B(y_0,d)\subset\{(y',y_n)\in \mathbb{R}^n;\; x_n<0\}$. We still denote, for notational convenience, the new matrix obtained after this translation and this change of coordinates by $A$. In that case, according to Lemma \ref{lemmaPCH}, $A$ belongs to the neighborhood $\mathcal{N}$ appearing in this lemma. Whence, $\mbox{supp}(u(\cdot, t))\cap B(z,\rho)\subset H_+$, for some $\rho>0$, with
\[
H_+=\left\{ (x',x_n)\in \mathbb{R}^{n-1}\times \mathbb{R};\; x_n\ge 0\right\}.
\]
Let $\varphi$ given by \eqref{PCH0} with $\vartheta =0$ and $\tilde{A}=A_H$. If $v$ is defined in a neighborhood $\tilde{\omega}$ by $v(y,\cdot)=u(\varphi^{-1}(y),\cdot)$ then straightforward computations give $\mathcal{L}^w_{\tilde{A}}v=0$ in $\tilde{\omega}$ and $\mbox{supp}(v(\cdot ,t))\subset E_+$ with
\[
E_+=\{ (y',y_n)\in \tilde{\omega};\; 0<y_n<1, y_n\ge |y'|^2\}.
\]
We complete the proof similarly to that of Theorem \ref{theoremUCw1}, with $\tilde{x}=0$, $B(0,r)\Subset \tilde{\omega}$, $\varkappa =1/4$ and $c=1$. Note that in the present case
\[
\psi (y,t)=(y_n-1)^2/2+|y'|^2/2-\gamma t^2/2+C,
\]
We get that there exists a universal constant $\mathfrak{t}^\ast$ so that, for any $\mathfrak{t}\ge  \mathfrak{t}^\ast$, we find $0<\tau \le \mathfrak{t}$ for which 
$v=0$ in $\tilde{\mathcal{U}}\times (-\tau,\tau)$, for some neighborhood $\mathcal{\tilde{U}}$ of $0$. In consequence $u=0$ in $\mathcal{U}\times (-\tau,\tau)$, where $\mathcal{U}$ is a neighborhood of $z$ in $\Omega$. In other words we proved that $u=0$ in $\mathcal{O}_0\times (-\tau ,\tau)$, where $\mathcal{O}_0=\mathcal{O}\cup \mathcal{U}\supsetneqq \mathcal{O}$. The proof is then complete.
\end{proof}

The result of Theorem \ref{theoremUCw2} is false without the condition that $\mathfrak{t}$ is sufficiently large as shows the non uniqueness result in \cite{AB1995MathZ}. The authors show that, in the two dimensional case, there exists $\mathcal{U}$, a neighborhood of the origin in $\mathbb{R}^2\times \mathbb{R}$, $p\in C^\infty (\mathcal{U})$, $u\in C^\infty (\mathcal{U})$ so that $(\Delta -\partial_t^2+p(x,t))u=0$ in $\mathcal{U}$ and $\mbox{supp}(u)\subset \mathcal {U}\cap \{(x_1,x_2,t)\in \mathbb{R}^2\times \mathbb{R};\; x_2\ge 0\}$.

A better result than that in Theorem \ref{theoremUCw2} can be obtained in the case of operators with time-independent coefficients. Let $\dot{\mathcal{L}}_A^w$ be the operator $\mathcal{L}_A^w$ when $q_0=0$, $q_j=q_j(x)$, $j=1,\ldots ,n$, and $p=p(x)$. 

\begin{theorem}\label{theoremUCw2-bis}
$($\cite{Ro1991CPDE}$)$ There exits $\mathfrak{t}^\ast=\mathfrak{t}(\Omega ,\varkappa, \mathfrak{m})$ so that, for any $\mathfrak{t}>\mathfrak{t}^\ast$ and $\omega \Subset \Omega$, if $u\in H^2(Q)$ satisfies $\dot{\mathcal{L}}_A^w=0$ in $Q$ and $u=0$ in $\omega \times (-\mathfrak{t},\mathfrak{t})$ then $u=0$ in $\Omega \times (-\tau,\tau)$, where $\tau =\mathfrak{t}-\mathfrak{t}^\ast$.
\end{theorem}

The main idea in the proof of Theorem \ref{theoremUCw2-bis} consists in transforming, via the Fourier-Bros-Iagolnitzer transform, the wave operator $\dot{\mathcal{L}}_A^w$ into an elliptic operator for which uniqueness of continuation property result is known.

Theorem \ref{theoremUCw2} can be used to establish a result on local uniqueness of continuation from Cauchy data on a part of the boundary.

\begin{corollary}\label{corollaryUCw1}
Let $\mathfrak{t}^\ast$ and $\mathcal{N}$ be as in Theorem \ref{theoremUCw2} with $\Omega$ substituted by larger domain $\hat{\Omega}\Supset \Omega$. Let $\Gamma_0$ be a nonempty open subset of $\Gamma$ and $\Sigma_0=\Gamma_0\times (-\mathfrak{t},\mathfrak{t})$ with $\mathfrak{t}\ge \mathfrak{t}^\ast$. There exist $\mathcal{U}$ a neighborhood of a point of $\Gamma_0$ in $\Omega$ and $0<\tau \le \mathfrak{t}$ so that if  $A\in \mathcal{N}$, and if $u\in H^2(Q;\mathbb{C})$ satisfies $\mathcal{L}_A^wu=0$ in $Q$ and $u=\partial_\nu u=0$ on $\Sigma_0$ then $u=0$ in $\mathcal{U}\times (-\tau ,\tau)$.
\end{corollary}

\begin{proof}
Pick $A\in \mathcal{N}$ and  $u\in H^2(Q;\mathbb{C})$ satisfying $\mathcal{L}_A^wu=0$ in $Q$ and $u=\partial_\nu u=0$ on $\Sigma_0$. Then there exists  $\mathcal{V}\subset \hat{\Omega}$, a neighborhood of a point in $\Gamma_0$, so that $\hat{u}$, the extension of $u$ by zero in $\mathbb{R}^n\setminus \overline{\Omega}$, belongs to $H^2(\Omega'\times (-\mathfrak{t},\mathfrak{t}))$, with $\Omega'=\Omega \cup \mathcal{V}$, satisfies $\mathcal{L}_A^w\hat{u}=0$ in $\Omega'\times (-\mathfrak{t},\mathfrak{t})$ and $\hat{u}=0$ in $[(\Omega'\setminus \overline{\Omega}) \cap \mathcal{V}]\times (-\mathfrak{t},\mathfrak{t})$. Theorem \ref{theoremUCw2} allows us to conclude that there exist $\mathcal{U}$, a neighborhood of a point of $\Gamma_0$ in $\Omega$, and $0<\tau \le \mathfrak{t}$ so that $u=0$ in $\mathcal{U}\times (-\tau ,\tau)$.
\end{proof}

It is worth mentioning that the following global unique continuation result from boundary data can be deduced from \cite[Theorem 1.1]{BC2019Arxiv}.

\begin{theorem}\label{UCwg}
Let $\Gamma_0$ an arbitrary non empty open subset of $\Gamma$. There exits $\mathfrak{t}^\ast=\mathfrak{t}(\Omega ,\varkappa, \mathfrak{m})$ so that, for any $\mathfrak{t}>\mathfrak{t}^\ast$ we find $0<\mathfrak{t}_0<\mathfrak{t}$ with the property that if $u\in C^\infty (\overline{\Omega}\times [-\mathfrak{t},\mathfrak{t}])$ satisfies $\mathcal{L}_{A,0}^wu=0$  and 
\[
u=\partial_\nu u=0\quad  \mbox{on}\quad \{\Gamma_0\times (-\mathfrak{t},\mathfrak{t})\}\cup \{\Gamma \times [(-\mathfrak{t},-\mathfrak{t}_0)\cup (\mathfrak{t}_0,\mathfrak{t})]\}
\]
then $u$ is identically equal to zero.
\end{theorem}

We end this subsection by remarking that we can proceed similarly to Theorem \ref{theoremUCw2} to prove the  unique continuation property across  a pseudo-convex hypersurface.

\begin{theorem}\label{theoremUCw3}
Let $A\in \mathscr{M}(\Omega ,\varkappa ,\mathfrak{m})$ and $H=\{x\in \omega ;\;\theta (x)=\theta (\tilde{x})\}$  be a $A$-pseudo-convex hypersurface defined in a neighborhood $\omega$ of $\tilde{x}\in \Omega$ with $\theta \in C^{3,1}(\overline{\omega})$. Then there exist $\mathcal{B}$, a neighborhood of $\tilde{x}$, and $\mathfrak{t}^\ast>0$ so that, for each $\mathfrak{t}\ge \mathfrak{t}^\ast$, we find $0<\tau \le \mathfrak{t}$ with the property that if $u\in H^2(\omega \times (-\mathfrak{t},\mathfrak{t}))$ satisfies $\mathcal{L}_A^wu=0$ in $\omega\times (-\mathfrak{t},\mathfrak{t})$ and 
$\mbox{supp}(u(\cdot ,t))\subset H_+=\{ x\in \omega ;\; \theta (x)\ge \theta (\tilde{x})\}$, $t\in  (-\mathfrak{t},\mathfrak{t})$,  then $u=0$ in $\mathcal{B}\times (-\tau,\tau)$.
\end{theorem}

\subsection{Observability inequality}

We suppose in this subsection that $t_1=0$ and $t_2=\mathfrak{t}>0$.

We shall need in the sequel the following technical lemma.
\begin{lemma}\label{lemmaOIw1}
Fix $0<\alpha <1$ and let $0\le \psi_0\in C^1(\overline{\Omega})$ satisfying 
\[
\min_{x\in \overline{\Omega}}|\nabla \psi_0 |_A^2:=\delta_0>0.
\]
Let $\mathbf{m}=\|\psi_0\|_{L^\infty(\Omega)}$ and define, for an arbitrary constant $C>0$,
\begin{equation}\label{OIw0}
\psi(x,t)=\psi_0(x)-\mathfrak{t}^{-2+\alpha}(t-\mathfrak{t}/2)^2+C,\quad x\in \overline{\Omega},\; t\in [0,\mathfrak{t}].
\end{equation}
If  $\mathfrak{t}> \mathfrak{t}_\alpha =\max\left( \delta_0^{-1/[2(1-\alpha)]}, (64\mathfrak{m}/2)^{1/\alpha}\right)$ then
\begin{align}
&\min_{\overline{Q}}\left(|\nabla \psi|_A^2-(\partial_t\psi)^2\right)^2:=\delta >0, \label{OIw1}
\\
&\psi(x,t) \ge - \mathfrak{t}^\alpha /64+C,\quad (x,t)\in \overline{\Omega}\times [3\mathfrak{t}/8,5\mathfrak{t}/8], \label{OIw2}
\\
&\psi (x,t)\le -2\mathfrak{t}^\alpha/64+C,\quad (x,t)\in \overline{\Omega}\times \left( [0,\mathfrak{t}/4]\cup [3\mathfrak{t}/4,\mathfrak{t}]\right).\label{OIw3}
\end{align}
\end{lemma}

\begin{proof}
If $\mathfrak{t}\ge \mathfrak{t}_\alpha$ then 
\[
|\nabla \psi_0|^2-(\partial_t\psi)^2\ge \delta_0 -\mathfrak{t}^{-4+2\alpha} \mathfrak{t}^2=\delta_0-\mathfrak{t}^{-2(1-\alpha)}:=\delta >0.
\]
That is we proved \eqref{OIw1}.

Inequality \eqref{OIw2} is straightforward. On the other hand,  we have
\[
\psi(x,t) \le \mathbf{m}-\mathfrak{t}^\alpha/16+C< 2\mathfrak{t}^\alpha /64 -\mathfrak{t}^\alpha /16+C=-2\mathfrak{t}^\alpha/64+C
\]
if $(x,t)\in \overline{\Omega}\times \left( [0,\mathfrak{t}/4]\cup [3\mathfrak{t}/4,\mathfrak{t}]\right)$. That is we proved \eqref{OIw3}.
\end{proof}

We consider in this subsection the following wave operator
\[
\mathcal{L}_A^w=\Delta_A-\partial_t^2 +p\partial_t +q, 
\]
with $p,q\in L^\infty(\Omega;\mathbb{C})$. We associate  to $\mathcal{L}_A^w$ the IBVP
\begin{equation}\label{OIw4}
\left\{
\begin{array}{ll}
\mathcal{L}_A^w=0\quad \mbox{in}\; Q,
\\
(u(\cdot, 0),\partial_tu(\cdot ,0))=(u_0,u_1),
\\
u_{|\Sigma}=0.
\end{array}
\right.
\end{equation}
According to the semigroup theory, for all $(u_0,u_1)\in H_0^1(\Omega )\times L^2(\Omega )$, the IBVP \eqref{OIw4} admits unique solution 
\[
u\in C([0,\mathfrak{t}];H_0^1(\Omega ))\cap C^1([0,\mathfrak{t}];L^2(\Omega )).
\]
We also know that $\partial_\nu u\in L^2(\Sigma )$ (hidden regularity).

From usual energy estimate for wave equations,  if
\begin{equation}\label{OIw5}
\mathcal{E}_u(t)=\|\mathbf{D}_Au(\cdot ,t)\|_{L^2(\Omega;\mathbb{C}^{n+1})},\quad 0\le t\le \mathfrak{t},
\end{equation}
then
\begin{equation}\label{OIw6}
\mathcal{E}_u(t)\le \aleph_0 \mathcal{E}_u(0),\quad 0\le t\le \mathfrak{t},
\end{equation}
where the constant $\aleph_0>0$  only depends of $\Omega$, $A$, $\mathfrak{t}$, $p$ and $q$.

We apply \eqref{OIw6} to $v(\cdot ,t)=u(\cdot ,s-t)$, with fixed $0<t\le s$. We find
\[
\mathcal{E}_u(s-t)=\mathcal{E}_v(t)\le \aleph_1 \mathcal{E}_v(0)=\aleph_1\mathcal{E}_u(s),\quad 0\le t\le s,
\]
where the constant $\aleph_1>0$  only depends of $\Omega$, $A$, $\mathfrak{t}$, $p$ and $q$. We have in particular 
\begin{equation}\label{OIw6.1}
\mathcal{E}_u(0)\le\aleph_1 \mathcal{E}_u(s),\quad 0\le t\le s.
\end{equation}
In light of \eqref{OIw6} and \eqref{OIw6.1} we get
\begin{equation}\label{OIw6.2}
\aleph ^{-1}\mathcal{E}_u(0)\le \mathcal{E}_u(t)\le \aleph \mathcal{E}_u(0),\quad 0\le t\le \mathfrak{t},
\end{equation}
for some constant $\aleph >1$ only depending of $\Omega$, $A$, $\mathfrak{t}$, $p$ and $q$.

We recall that $\Sigma_+=\Gamma_+\times (0,\mathfrak{t})$.

\begin{theorem}\label{theoremOIw1}
Fix $0<\alpha<1$ and assume that $0\le \psi_0\in C^4(\overline{\Omega})$ is $A$-pseudo-convex with constant $\kappa >0$ and let $\Gamma_+=\{x\in \Gamma ;\; \partial_{\nu_A}\psi_0(x)>0\}$. If $\tilde{\mathfrak{t}}_\alpha=\min \left(\mathfrak{t}_\alpha, (8\varkappa/\kappa)^{1/(2-\alpha)}\right)$, where $\mathfrak{t}_\alpha$ be as in Lemma \ref{lemmaOIw1}, then, for any $\mathfrak{t}\ge \tilde{\mathfrak{t}}_\alpha $ and $(u_0,u_1)\in H_0^1(\Omega)\times L^2(\Omega )$, we have
\[
\|(u_0,u_1)\|_{H_0^1(\Omega)\times L^2(\Omega )}\le \aleph \|\partial_\nu u\|_{L^2(\Sigma_+)},
\]
where the constant $\aleph>0$ only depends of  $\Omega$, $\mathfrak{t}$, $\varkappa$, $\kappa$, $\Gamma_+$ and  $u$ is the solution of the IBVP \eqref{OIw4} corresponding to $(u_0,u_1)$.
\end{theorem}

\begin{proof}
 Fix $\mathfrak{t}>\tilde{\mathfrak{t}}_\alpha$ and let $\psi$ defined as in \eqref{OIw0} in which the constant $C>0$ is chosen sufficiently large to guarantee that $\psi \ge 0$. In that case we easily check that $\phi=e^{\lambda \psi}$ is a weight function for the operator $\mathcal{L}_{A,0}^w$ in $Q$.

Pick $\varrho\in C_0^\infty (\mathfrak{t}/8,7\mathfrak{t}/8)$ so that $\varrho =1$ in $[\mathfrak{t}/4,3\mathfrak{t}/4]$. Clearly, a simple density argument shows that Corollary \ref{CarlemanCorollary2} remains valid for $\varrho u$ for any solution $u$ of \eqref{OIw4}  with $(u_0,u_1)\in H_0^1(\Omega )\times L^2(\Omega )$. According to this Corollary we have, for fixed $\lambda \ge \lambda^\ast$ and any $\tau \ge \tau^\ast$,
\begin{align}
\aleph\int_Qe^{2\tau \phi} &|\mathbf{D}_A (\varrho u)|^2dxdt  \label{OIw7}
\\
&\le \int_Qe^{2\tau \phi}\left|\mathcal{L}_A^w(\varrho u)\right|^2dxdt +\int_{\Sigma_+} e^{2\tau \phi}|\partial_\nu (\varrho u)|^2 d\sigma dt,\nonumber
\end{align}
But
\[
\mathcal{L}_A^w(\varrho u)=2\varrho 'u+\varrho ''u.
\]
Hence \eqref{OIw7} together with Poincar\'e's inequality ($u(\cdot, t)\in H_0^1(\Omega )$) give
\begin{align}
\aleph\int_Qe^{2\tau \phi} &|\mathbf{D}_A (\varrho u)|^2dxdt  \label{OIw8}
\\
&\le \int_{Q\cap \rm{supp}(\varrho ')}e^{2\tau \phi}|\mathbf{D}_A u|^2dxdt +\int_{\Sigma_+} e^{2\tau \phi}|\partial_\nu u|^2 d\sigma dt.\nonumber
\end{align}
Define
\[
c_0 =e^{\lambda (-\gamma \mathfrak{t}^2/64+C)}\quad \mbox{and}\quad c_1 =e^{\lambda (-2\gamma \mathfrak{t}^2/64+C)}.
\]
If $\mathcal{E}_u$ is given  by \eqref{OIw5} then we get from \eqref{OIw2}, \eqref{OIw3} and \eqref{OIw8}, where $\tau \ge \tau^\ast$, 
\[
\aleph e^{\tau c_0}\int_{3\mathfrak{t}/8}^{5\mathfrak{t}/8} \mathcal{E}_u(t)dt  \le e^{\tau c_1} \int_0^\mathfrak{t}\mathcal{E}_u(t)dt +\int_{\Sigma_+} e^{2\tau \phi}|\partial_\nu u|^2 d\sigma dt,
\]
This inequality together with \eqref{OIw6.2} imply
\[
\left(\aleph e^{\tau c_0} -e^{\tau c_1}\right)\mathcal{E}_u(0) \le \int_{\Sigma_+} e^{2\tau \phi}|\partial_\nu u|^2 d\sigma dt,\quad \tau \ge \tau^\ast.
\]
As $c_0 >c_1$, we fix $\tau$ sufficiently large in such a way that $\tilde{\aleph}=\aleph e^{\tau c_0} -e^{\tau c_1}>0$. That is we have
\begin{equation}\label{OIw9}
\tilde{\aleph}\mathcal{E}_u(0) \le \int_{\Sigma_+} e^{2\tau \phi}|\partial_\nu u|^2 d\sigma dt,\quad \tau \ge \tau^\ast.
\end{equation}
The expected inequality follows readily from \eqref{OIw9}.
\end{proof}

From the calculations in Example \ref{example1} when $\psi_0(x)=|x-x_0|^2/2$, with $x_0\in \mathbb{R}^n\setminus\overline{\Omega}$, there exists a neighborhood $\mathcal{N}$ of $\mathbf{I}$ in $C^{2,1}(\overline{\Omega};\mathbb{R}^{n\times n})$ so that that for any $A\in \mathcal{N}$, $\varkappa=1/2$ and $\psi_0$ is $A$-pseudo-convex with constant $\kappa=1/4$. In this case
\[
\tilde{\mathfrak{t}}_\alpha=\max\left( d_0^{-1/(1-\alpha)},[16(d+d_0)]^{1/\alpha},16^{1/(2-\alpha)}\right),
\]
with $d_0=\mbox{dist}(x_0,\overline{\Omega})$ and $d=\mbox{diam}(\Omega )$.

A result in the variable coefficients case was already established in \cite[Theorem 1.1]{Ya1999SICON}. This result is based on a generalization of the multiplier method in which a vector field is used as an alternative to the multiplier. This vector field satisfies a certain convexity condition. Note however that the lower bound in $\mathfrak{t}$ appearing in \cite[Theorem 1.1]{Ya1999SICON} is not easily comparable to that we used in Theorem \ref{theoremOIw1}.  The minimal time guaranteeing observability was estimated in precise way in \cite{DZZ2008} for wave equations with $C^1$ variable coefficients. Recently an observability result was established in \cite{Sh2019} for constant coefficients wave equation in the case of time-dependent domains. The minimal time in \cite{Sh2019} is explicit.

\section{Elliptic equations}

We show briefly how we can modify the calculations we carried out for wave equations in order to retrieve Carleman inequalities for elliptic equations and the corresponding property of unique continuation. In  this section
\[
\mathcal{L}_A^e=\Delta _A+\sum_{\ell=1}^np_\ell \partial_\ell +q,
\]
where $p_1,\ldots p_n$ and $q$ belong to $L^\infty (\Omega ;\mathbb{C})$ and satisfy
\[
\|p_\ell\|_{L^\infty (\Omega)}\le \mathfrak{m},\; 1\le \ell\le n\quad \mbox{and}\quad \|q\|_{L^\infty (\Omega)}\le \mathfrak{m}.
\]
Also, $0\le \psi\in C^4(\overline{\Omega})$ is fixed so that
\[
|\nabla \psi|\ge \delta \quad \mbox{in}\; \overline{\Omega},
\]
for some constant $\delta >0$.

\subsection{Carleman inequality}

Let $\phi=e^{\lambda \psi}$ and set $\mathfrak{d}=(\Omega ,\varkappa, \delta ,\mathfrak{m} )$.

\begin{theorem}\label{CarlemanTheoremEll1}
We find three constants $\aleph=\aleph(\mathfrak{d})$, $\lambda^\ast=\lambda^\ast(\mathfrak{d})$ and $\tau^\ast =\tau^\ast(\mathfrak{d})$ so that, for any $\lambda \ge \lambda^\ast$, $\tau \ge \tau^\ast$ and $u\in H^2(\Omega ,\mathbb{C})$, we have
\begin{align}
&\aleph\int_\Omega e^{2\tau \phi}\left[\tau^3\lambda ^4\phi^3 |u|^2+\tau \lambda^2  \phi |\nabla u|^2\right]dx \label{C29}
\\
&\hskip1cm \le \int_\Omega e^{2\tau \phi}|\mathcal{L}_A^eu|^2dx +\int_\Gamma e^{2\tau \phi}\left[\tau^3\lambda ^3\phi^3|u|^2+\tau \lambda \phi |\nabla u|^2\right]d\sigma .\nonumber
\end{align}
\end{theorem}

\begin{proof}
In this proof, $\lambda _k$ and $\tau_k$, $k=1,2,\ldots$, denote generic constants only depending on $\mathfrak{d}$.

Let $\Phi =e^{-\tau \phi}$, $\tau >0$. We have from the calculations of the preceding section 
\[
L=\Phi^{-1}\Delta_A(\Phi w)=\Delta_Aw -2\tau (\nabla w|\nabla \phi)_A+\left[\tau^2  |\nabla \phi|_A^2-\tau\Delta_Aw\right]w.
\]
We decompose $L$ in the following special form
\[
L=L_0+L_1+c,
\]
with, for $w\in H^2(\Omega ,\mathbb{R})$,
\begin{align*}
&L_0w=\Delta_Aw+aw,
\\
&L_1w=(B|\nabla w)+bw. 
\end{align*}
The coefficients of $L_0$ and $L_1$ and $c$ are given as follows
\begin{align*}
&a=\tau^2|\nabla \phi|_A^2,
\\
&b=-2\tau \Delta_A\phi,
\\
&c=\tau \Delta_A\phi ,
\\
&B=-2\tau A\nabla \phi .
\end{align*}
We have
\begin{equation}\label{el1}
\langle L_0|L_1\rangle_{L^2(\Omega)}=\sum_{k=1}^4I_k,
\end{equation}
where
\begin{align*}
&I_1=\int_\Omega \Delta_A w(\nabla w|B) dx,
\\
&I_2=\int_\Omega \Delta_A wbwdx,
\\
&I_3=\int_\Omega aw(\nabla w|B) dx,
\\
&I_4=\int_\Omega abw^2dx.
\end{align*}
Let 
\[
D=C/2-A(B')^t.
\]
Then straightforward modifications of the computations of the preceding section yield
\begin{align}
&I_1=\int_\Omega (D\nabla w|\nabla w)dx+\int_\Gamma \left[(\nabla w|\nu)_A(\nabla w|B)-(B/2|\nu)|\nabla w|_A^2\right]d\sigma ,\label{el2}
\\
&I_2=-\int_\Omega b|\nabla w|_A^2dx +\int_\Omega \Delta_A(b/2) w^2 dx \label{el3}
\\
&\hskip 4cm -\int_\Gamma (\nabla (b/2)|\nu)_Aw^2d\sigma +\int_\Gamma (\nabla w|\nu )_Abw d\sigma ,\nonumber
\\
&I_3=-\int_\Omega \mbox{div}(aB/2)w^2dx+ \int_\Gamma a(B/2|\nu)w^2d\sigma .\label{el4}
\end{align}
Identities \eqref{el2} to \eqref{el4} in \eqref{el1} give
\begin{equation}\label{el5}
\langle L_0|L_1\rangle_{L^2(\Omega)}=\int_\Omega (\mathfrak{A}\nabla w|\nabla w)dx+\int_\Omega \mathfrak{a}w^2dx+\int_\Gamma g(w) d\sigma,
\end{equation}
with
\begin{align*}
&\mathfrak{A}=D-bA,
\\
&\mathfrak{a}=ab+\Delta_A(b/2)-\mbox{div}(aB/2),
\\
&g(w)=(\nabla w|\nu)_A(\nabla w|B)-(B/2|\nu)|\nabla w|_A^2
\\
&\hskip 4cm-(\nabla (b/2)|\nu)_Aw^2+(\nabla w|\nu )_Abw+a(B/2|\nu)w^2.
\end{align*}
We have
\[
(\mathfrak{A}\xi|\xi)=\tau \lambda^2\left[|\nabla \psi|^2(A\xi|\xi)+(\nabla \psi |A\xi)^2\right]+\tau \lambda (\tilde{\mathfrak{A}}\xi|\xi),
\]
where $\tilde{\mathfrak{A}}$ is a matrix depending only on $A$ and $\psi$. Therefore
\begin{equation}\label{el6}
(\mathfrak{A}\xi|\xi)\ge \tau \lambda^2\delta ^2\varkappa |\xi |^2/2,\quad \lambda \ge \lambda_1.
\end{equation}
We have also
\[
\mathfrak{a}=\tau^3\lambda ^4\phi^3|\nabla \psi |_A^4+\tilde{\mathfrak{a}},
\]
where the reminder term $\tilde{\mathfrak{a}}$ contains, as for the wave equation, only terms with factors $\tau ^k\lambda ^\ell\phi^m$, $1\le k,\ell,m \le 3$ and terms with factor $\tau \lambda^4\phi$. Hence
\begin{equation}\label{el7}
\mathfrak{a}\ge \tau^3\lambda ^4\delta ^4\phi^3/2,\quad \lambda \ge \lambda_2,\; \tau \ge \tau_2.
\end{equation}
The rest of the proof is almost similar to that of the wave equation.
\end{proof}

\begin{remark}\label{remarkE1}
{\rm The symbol of the principal part of the operator $\mathcal{L}_A^e$ is given by
\[
p(x,\xi)=|\xi|_A^2(x)=(A(x)\xi|\xi),\quad x\in \overline{\Omega},\; \xi\in \mathbb{R}^n.
\]
Therefore if $\phi \in C^4(\overline{\Omega})$ then we have $p(x,\xi+i\tau \nabla \phi)=p_0+ip_1$ with
\[
p_0=|\xi|_A^2-\tau^2|\nabla \phi |_A,\quad p_1=2\tau (\xi|\nabla \phi)_A.
\]
When $\phi =e^{\lambda \psi}$ we find, for $\tau \ge 1$,
\[
\{p_0,p_1\}:=\sum_{j=1}^n\left[\partial_{\xi_j}p_0\partial_{x_j}p_1-\partial_{x_j}p_0\partial_{x_j}p_1\right]=\tau ^2\left[2\lambda ^3\phi^2|\nabla \psi|_A^4+O(\lambda ^2)\right].
\]
In consequence $\phi$ satisfies the sub-ellipticity condition in \cite[Theorem 8.3.1, page 190]{Ho1976Springer} if $\lambda$ is sufficiently large and hence the following Carleman inequality holds: there exist $\aleph >0$ and $\tau^\ast>0$ only depending on $\Omega$ and bounds on the coefficients of $\mathcal{L}_A^e$ so that
\[
\sum_{|\alpha|\le 1}\tau^{2(2-|\alpha|)}\int_\Omega e^{2\tau \phi}|\partial^\alpha u|^2dx\le \aleph \tau \int_\Omega e^{2\tau \phi}|\mathcal{L}_A^eu|^2dx,\quad u\in C_0^\infty (\Omega),\; \tau \ge \tau ^\ast.
\]
In other words, if $\phi=e^{\lambda \psi}$ is a weight function for the elliptic operator $\mathcal{L}_{A,0}^e$ then $\phi$ possesses the sub-ellipticity condition for large $\lambda$.
}
\end{remark}

\subsection{Unique continuation}

We use a similar method as for the wave equation. For sake of completeness, we provide some details.

We start with a unique continuation result across a convex hypersurface. To this end, we set 
\[
\psi(x',x_n)=(x_n-1)^2+|x'|^2,\quad (x',x_n)\in \mathbb{R}^{n-1}\times \mathbb{R}.
\]
As $|\nabla \psi (0,0)|=2$, there exists $r>0$ so that $|\nabla \psi|\ge 1$ in $B(0,r)$. Consider then the set
\[
E_+=\left\{(x',x_n)\in B(0,r) ;\; 0\le x_n<1\; \mbox{and}\; x_n\ge |x'|^2\right\}.
\]
We have clearly
\[
E_+\setminus\{(0,0)\}=\left\{ (x',x_n)\in \mathbb{R}^{n-1}\times \mathbb{R};\; \psi (x',x_n)<\psi(0,0)=1\right\}.
\]
Pick $\chi \in C_0^\infty (B(0,r))$ satisfying $\chi =1$ in $B(0,\rho_1)$, where $0<\rho_1 <r$ is fixed arbitrary. Let then $\epsilon >0$ so that
\[
E_+\cap \left[B(0,r)\setminus\overline{B}(0,\rho_1)\right]\subset \left\{ (x',x_n)\in B(0,r);\; \psi (x',x_n)<\psi(0,0)-\epsilon \right\}.
\]

\begin{lemma}\label{lemmaUCE1}
There exists $0<\rho_0<\rho_1$ so that if $u\in H^2(B(0,r);\mathbb{C})$ satisfies $\mathcal{L}_A^eu=0$ in $B(0,r)$ and $\mbox{supp}(u)\subset E_+$ then $u=0$ in $B(0,\rho_0)$.
\end{lemma}

\begin{proof}
Let us choose $0<\rho_0 <\rho_1$  in such a way that 
\[
E_+\cap B(0,\rho_0)\subset \left\{ (x',x_n)\in B(0,r);\; \psi (x',x_n)>\psi (0,0)-\epsilon /2\right\}.
\]
Pick $u\in H^2(B(0,r);\mathbb{C})$ satisfying $\mathcal{L}_A^eu=0$ in $B(0,r)$ and $\mbox{supp}(u)\subset E_+$. Let $v=\chi u$, and $\lambda ^\ast$ and $\tau^\ast$ be as in Theorem \ref{CarlemanTheoremEll1}. Fix then $\lambda \ge \lambda ^\ast$ and set
\[
c_0 =e^{\lambda (1-\epsilon/2)}, \quad c_1 = e^{\lambda (1-\epsilon)} .
\]
Theorem \ref{CarlemanTheoremEll1} yields
\begin{align*}
\int_{B(0,\rho_0)}u^2dx &=\int_{B(0,\rho_0)\cap E_+}v^2dx 
\\
&\le \aleph \tau^{-3} e^{-\tau (c_0 -c_1)}\int_{E_+\cap(B(0,r)\setminus B(0,\rho_1))}(\mathcal{L}_A^ev)^2dx,\quad \tau \ge \tau^\ast .
\end{align*}
Noting that $c_0 >c_1$,  we obtain that $u=0$ in $B(0,\rho_0)$ by taking in the right hand side of the last inequality the limit, as $\tau$ tends to $\infty$.
\end{proof}

Let $\vartheta =\vartheta (x')$ be in $C^{3,1}(\overline{B}(0,r))$ satisfying $\vartheta (0)=0$ and $\nabla'\vartheta (0)=0$, and consider 
\[
\varphi :(x',x_n)\in \omega \mapsto (y',y_n)=(x',x_n-\vartheta (x')+|x'|^2).
\]
As $\varphi '(0,0)=\mathbf{I}$, we deduce that $\varphi$ is a diffeomorphism from $\omega \subset B(0,r)\times \mathbb{R}$, a neighborhood of $0$ in $\mathbb{R}^n$, onto $\tilde{\omega}=\varphi(\omega)$.

Pick $u\in H^2(\omega ,\mathbb{C})$ satisfying $\mathcal{L}_A^eu=0$ in $\omega$ and $\mbox{supp}(u)\subset \omega_+=\{(x',x_n)\in \omega ;\; x_n\ge \vartheta(x')\}$. Define $v$ by $v(y',y_n)=u(\varphi^{-1}(y',y_n))$, $(y',y_n)\in \tilde{\omega}$. Then it is straightforward to check that $\mathcal{L}_{\tilde{A}}^e v=0$ in $\tilde{\omega}$. Here $\mathcal{L}_{\tilde{A}}^e$ is of the same form as $\mathcal{L}_A^e$. Its principal part is given by
\[
\mathcal{L}_{\tilde{A},0}^e= \Delta_{\tilde{A}}
\]
with
\[
\tilde{A} (y)=\varphi'\left(\varphi ^{-1}(y)\right)A\left(\varphi ^{-1}(y)\right)(\varphi')^t\left(\varphi ^{-1}(y)\right).
\]
Furthermore, $\mbox{supp}(v)\subset \tilde{\omega}_+=\{ (y',y_n)\in \tilde{\omega};\; y_n\ge |y'|^2\}$.

Similar calculations as in Subsection \ref{subsectionPCH} show, by reducing $\omega$ if necessary, that
\[
\left(\tilde{A} (y)\xi |\xi\right)\ge \varkappa |\xi|^2/4,\quad y\in \tilde{\omega},\; \xi \in \mathbb{R}^n.
\]
We apply Lemma \ref{lemmaUCE1} in order to get $v=0$ in $\tilde{\mathcal{V}}$, where $\tilde{\mathcal{V}}$ is a neighborhood of the origin, and hence $u=0$ in $\mathcal{V}$ with $\mathcal{V}=\varphi^{-1}(\tilde{\mathcal{V}})$. In other words, we proved the following result.

\begin{lemma}\label{lemmaUCE2}
There exists a neighborhood $\mathcal{V}$ of the origin in $\omega$ so that if  $u\in H^2(\omega;\mathbb{C} )$ satisfies $\mathcal{L}_A^eu=0$ in $\omega$ and $\mbox{supp}(u)\subset \omega_+$ then $u=0$ in $\mathcal{V}$.
\end{lemma}

The global uniqueness of continuation result is based on the following lemma.

\begin{lemma}\label{lemmaUCE3}
Let $\Omega_0\subset \Omega$ so that $\partial \Omega_0\cap \Omega\ne \emptyset$. There exists $z\in \partial \Omega_0\cap \Omega$ and $\mathcal{W}$ a neighborhood of $z$ in $\Omega$ so that if  $u\in H^2(\Omega;\mathbb{C})$ satisfies $\mathcal{L}_A^eu=0$ in $\Omega$ together with $u=0$ in $\Omega_0$ then $u=0$ in $\mathcal{W}$.
\end{lemma}

\begin{proof}
Fix $y\in \partial \Omega_0\cap \Omega$ and $r>0$ so that $B(y,r)\subset \Omega$. Pick then $y_0\in \Omega_0\cap B(y,r)$ sufficiently close to $y$ in such a way that $\partial B(y_0, d)\cap \partial \Omega_0\ne \emptyset$, with $d=\mbox{dist}(y_0,\partial \Omega_0)$. Pick then $z\in \partial B(y_0, d)\cap \partial \Omega_0$. Making a translation we may assume that $z=0$. As the $\partial B(y_0, d)$ can be represented locally by a graph $x_n=\vartheta (x')$. Making a change of coordinates we may assume that $\mbox{supp}(u)\subset \{(x',x_n);\; x_n\ge \vartheta (x')\}$. This orthogonal transformation modify $A$, but the new matrix has the same properties as $A$. We then complete the proof by using Lemma \ref{lemmaUCE2} with $A$ substituted by this new matrix.
\end{proof}

\begin{theorem}\label{theoremUCE1}
Let $u\in H^2(\Omega;\mathbb{C})$ satisfying $\mathcal{L}_A^eu=0$ in $\Omega$ and $u=0$ in $\omega$, for some nonempty open subset $\omega$ of $\Omega$. Then $u=0$ in $\Omega$. 
\end{theorem}

\begin{proof}
Let $\Omega_0$ be the maximal domain in which $u=0$. If $\Omega \setminus \overline{\Omega_0}\ne\emptyset$ then we would have $\partial \Omega_0\cap \Omega\ne \emptyset$. Therefore we would find, by Lemma \ref{lemmaUCE3}, $z\in \partial \Omega_0\cap \Omega$ and $\mathcal{W}$ a neighborhood of $z$ in $\Omega$ so that $u=0$ in $\mathcal{W}$. That is $u=0$ in $\Omega_0\cup \mathcal{W}$ which contains strictly $\Omega_0$ and hence contradicts the maximality of $\Omega_0$.
\end{proof}

It is worth mentioning that Theorem \ref{theoremUCE1} can also be obtained as a consequence \cite[Proposition 2.28, page 28]{Ch2016Springer} that quantifies the uniqueness of continuation from a subdomain of $\Omega$ to another subdomain of $\Omega$. The proof of \cite[Proposition 2.28, page 28]{Ch2016Springer} relies on three-ball inequality which is itself a consequence of the Carleman inequality of Theorem \ref{CarlemanTheoremEll1}.

As an immediate consequence of Theorem \ref{theoremUCE1} we get uniqueness of continuation from the Cauchy data on a subboundary.
\begin{corollary}\label{corollaryUCE1}
Assume that $\mathcal{L}_A^e$ is defined in $\hat{\Omega}\Supset \Omega$. Let $\Gamma_0$ be an arbitrary nonempty open subset of $\Gamma$. If $u\in H^2(\Omega;\mathbb{C})$ satisfies $\mathcal{L}_A^eu=0$ in $\Omega$ together with $u=\partial_\nu u= 0$ in $\Gamma_0$ then $u=0$ in $\Omega$. 
\end{corollary}

\begin{proof}
If $u\in H^2(\Omega ;\mathbb{C})$ satisfies $u=\partial_\nu u= 0$ in $\Gamma_0$ then we find $\mathcal{V}\subset \hat{\Omega}$, a neighborhood of a point in $\Gamma_0$, so that $\tilde{u}$, the extension of $u$ by zero in $\mathbb{R}^n\setminus \overline{\Omega}$, belongs to $H^2(\Omega \cup \mathcal{V})$, satisfies $\mathcal{L}_A^e\tilde{u}=0$ in $\Omega \cup \mathcal{V}$ and $\tilde{u}=0$ in $\mathcal{V}\setminus \overline{\Omega}$. Theorem \ref{theoremUCE1} allows us to conclude that $\tilde{u}=0$ and hence $u=0$.
\end{proof}

We observe once again that Corollary \ref{corollaryUCE1} can be deduced from \cite[Proposition 2.28 in page 28 and Proposition 2.30 in page 29]{Ch2016Springer}. We point out that \cite[Proposition 2.30 in page 29]{Ch2016Springer} quantifies the uniqueness of continuation from the Cauchy data on a  subboundary to an interior subdomain.

\section{Parabolic equations}

We fix in this section $0\le \psi\in C^4(\overline{Q})$ of the form $\psi(x,t)=\psi_0(x)+\psi_1(t)$, where
\[
|\nabla \psi_0|\ge \delta \quad \mbox{in}\; \overline{\Omega},
\]
for some constant $\delta >0$. Let $\phi=e^{\lambda \psi}$, $\lambda>0$, and consider the parabolic operator
\[
\mathcal{L}_A^p=\Delta _A-\partial_t+\sum_{\ell=1}^np_\ell \partial_\ell +q,
\]
where $A\in \mathscr{M}(\Omega ,\varkappa,\mathfrak{m})$, $p_1,\ldots p_n$ and $q$ belong to $L^\infty (Q ;\mathbb{C})$ and satisfy
\[
\|p_\ell\|_{L^\infty (Q)}\le \mathfrak{m},\; 1\le \ell\le n\quad \mbox{and}\quad \|q\|_{L^\infty (Q)}\le \mathfrak{m}.
\]

\subsection{Carleman inequality}

Recall that $\mathcal{L}_{A,0}^p$ represents the principal part of $\mathcal{L}_A^p$:
\[
\mathcal{L}_{A,0}^p=\Delta _A-\partial_t.
\]
We will use in the sequel the notation $\mathfrak{d}=(\Omega ,t_1,t_2,\varkappa, \delta ,\mathfrak{m} )$.

\begin{theorem}\label{CarlemanTheoremPar1}
There exist three constants $\aleph=\aleph(\mathfrak{d})$, $\lambda^\ast=\lambda^\ast(\mathfrak{d})$ and $\tau^\ast =\tau^\ast(\mathfrak{d})$ so that, for any $\lambda \ge \lambda^\ast$, $\tau \ge \tau^\ast$ and $u\in H^{2,1}(Q,\mathbb{C})$, we have
\begin{align}
&\aleph\int_Qe^{2\tau \phi}\left[\tau^3\lambda ^4\phi^3 |u|^2+\tau \lambda^2  \phi |\nabla u|^2\right]dxdt \label{C29}
\\
&\hskip 1cm \le \int_Qe^{2\tau \phi}|\mathcal{L}_A^pu|^2dxdt+\int_{\partial Q} e^{2\tau \phi}\left[\tau^3\lambda ^3\phi^3|u|^2+\tau \lambda \phi |\nabla u|^2\right]d\mu \nonumber
\\
&\hskip 3cm +\int_\Sigma e^{2\tau \phi}(\tau \lambda \phi)^{-1}|\partial_tu|^2d\sigma dt,\nonumber
\end{align}
\end{theorem}

\begin{proof}
As for the wave equation, we set $\Phi=e^{-\tau \phi}$, $\tau >0$. We also recall that
\begin{align*}
&\partial_k\Phi =-\tau \partial_k\phi \Phi  ,
\\
&\partial_{k\ell}\Phi=\left(-\tau \partial^2_{k\ell}\phi +\tau^2\partial_k\phi\partial_\ell\phi\right)\Phi,
\\
&\partial_t\Phi= -\tau \partial_t\phi\Phi  .
\end{align*}

We have, for $w\in H^{2,1}(Q,\mathbb{R})$,
\[
\Phi^{-1}\Delta_A(\Phi w)=\Delta_Aw -2\tau (\nabla w|\nabla \phi)_A+\left[\tau^2  |\nabla \phi|_A^2-\tau\Delta_Aw\right]w.
\]
Also,
\[
\Phi^{-1}\partial_t(\Phi w)=\partial_tw-\tau \partial_t\phi w.
\]
We decompose $L=\Phi^{-1}\mathcal{L}_{A,0}^p\Phi $ as in the elliptic case. That is in the form
\[
L=L_0+L_1+c,
\]
with
\begin{align*}
&L_0w=\Delta_A w+aw,
\\
&L_1w= (B|\nabla w)- \partial_tw +bw,
\end{align*}
where
\begin{align*}
&a(x,t)=\tau^2|\nabla \phi|_A^2,
\\
&b(x,t)=-2\tau \Delta_A  \phi ,
\\
&c(x,t)=\tau \Delta_A  \phi +\tau \partial_t\phi,
\\
&B=-2\tau  A\nabla\phi.
\end{align*}
We have
\begin{equation}\label{pa1}
\langle L_0w|L_1w\rangle_{L^2(Q)}=\sum_{j=1}^6 I_j.
\end{equation}
Quantities $I_j$, $1\le j\le 6$, are given as follows
\begin{align*}
&I_1=\int_Q\Delta_A w(\nabla w|B) dxdt,
\\
&I_2=-\int_Q\Delta_A w\partial_twdxdt,
\\
&I_3=\int_Q\Delta_A wbwdxdt,
\\
&I_4=\int_Qaw(\nabla w|B) dxdt,
\\
&I_5=-\int_Qaw\partial_twdxdt,
\\
&I_6=\int_Qabw^2dxdt.
\end{align*}
Let $C=(\mbox{div}(a_{k\ell}B))$ and
\[
D=C/2-A(B')^t.
\]
We already proved that
\begin{equation}\label{pa2}
I_1=\int_Q(D\nabla w|\nabla w)dxdt+\int_\Sigma \left[(\nabla w|\nu)_A(\nabla w\cdot B)-(B/2|\nu)|\nabla w|_A^2\right]d\sigma dt.
\end{equation}
On the other hand inequality \eqref{ani5} with $d=-1$ gives
\begin{equation}\label{pa3}
I_2=\int_\Sigma (\nabla w|\nu)_A \partial_twdxdt+\int_\Omega \left[ |\nabla w|_A^2/2\right]_{t=t_1}^{t_2}dx.
\end{equation}
$I_3$ is the same as in \eqref{ani6}:
\begin{align}
I_3=&-\int_Q b|\nabla w|^2dxdt +\int_Q \Delta_A(b/2) w^2 dxdt \label{pa4}
\\
&\hskip 1cm -\int_\Sigma (\nabla (b/2)|\nu)_Aw^2d\sigma dt +\int_\Sigma (\nabla w|\nu )_Abw d\sigma dt.\nonumber
\end{align}
Let $J_1=I_1+I_2+I_3$ and
\begin{align*}
&\mathfrak{A}=D-bA,
\\
&a_1=\Delta_A (b/2),
\\
&g_1(w)=(\nabla w|\nu)_A(\nabla w|B)-(B/2|\nu)|\nabla w|_A^2-(\nabla w|\nu)_A \partial_tw
\\
&\hskip 5cm -(\nabla (b/2)|\nu)_Aw^2d\sigma + (\nabla w|\nu )_Abw,
\\
&h_1(w)=\left[ |\nabla w|_A^2/2\right]_{t=t_1}^{t_2}.
\end{align*}

We find by putting together \eqref{pa2} to \eqref{pa4} 
\begin{align}
J_1=\int_Q (\mathfrak{A}\nabla w|\nabla w)dxdt &+\int_Qa_1w^2dxdt.\label{pa5}
\\
 &+\int_\Sigma g_1(w)d\sigma dt +\int_\Omega h_1(w) dx.\nonumber
\end{align}
 $I_4$ was calculated in \eqref{ani13}. Precisely, we have
\begin{equation}\label{pa6}
I_4=-\int_Q  \mbox{div}(aB/2)w^2dxdt + \int_\Sigma a(B/2|\nu)w^2d\sigma dt.
\end{equation}
On the other hand an integration by parts, with respect  to $t$, gives
\begin{equation}\label{pa7}
I_5=-\int_Q (a/2)\partial_t(w^2)dxdt= \int_Q\partial_t (a/2) w^2dxdt-\int_\Omega \left[(a/2)w^2\right]_{t=t_1}^{t_2}dx.
\end{equation}
Set
\begin{align*}
&a_2=-\mbox{div}(aB/2)+\partial_t(a/2)+ab,
\\
&g_2(w)=a(B/2|\nu)w^2,
\\
&h_2(w)=-\left[(a/2)w^2\right]_{t=t_1}^{t_2},
\end{align*}
and let $J_2=I_4+I_5+I_6$. In light of \eqref{pa6} and \eqref{pa7} we get
\begin{equation}\label{pa8}
J_2=\int_Q a_2w^2dxdt+\int_\Sigma g_2(w)d\sigma dt +\int_\Omega h_2(w) dx.
\end{equation}
Putting together \eqref{pa1}, \eqref{pa5} and \eqref{pa8} in order to obtain 
\begin{align}
&\langle L_0w|L_1w\rangle_{L^2(Q)}=2\tau \int_Q (\mathfrak{A}\nabla w|\nabla w)dxdt \label{pa9}
\\
&\hskip 4cm +\int_Q\mathfrak{a}w^2dxdt +\int_\Sigma g(w)d\sigma dt +\int_\Omega h(w) dx,\nonumber
\end{align}
where $\mathfrak{a}=a_1+a_2$, $g=g_1+g_2$ and $h=h_1+h_2$.

As $\partial_t a=0$ (which is a consequence of $\partial_t\nabla \psi=0$), we see that $\mathfrak{A}$ and $\mathfrak{a}$ has exactly the same form as in the elliptic case. Therefore we can mimic the proof of the elliptic case to complete the proof.
\end{proof}

We already defined  $\Gamma_+=\Gamma_+^{\psi_0}=\{x\in \Gamma ;\; \partial_{\nu_A} \psi_0>0\}$ and $\Sigma_+=\Sigma_+^{\psi_0}=\Gamma_+\times (t_1,t_2)$. Similarly to the wave equation we have the following result.

\begin{theorem}\label{CarlemanTheoremPar2}
There exist three constants $\aleph=\aleph(\mathfrak{d})$, $\lambda^\ast=\lambda^\ast(\mathfrak{d})$ and $\tau^\ast =\tau^\ast(\mathfrak{d})$ so that, for any $\lambda \ge \lambda^\ast$, $\tau \ge \tau^\ast$ and $u\in H^{2,1}(Q,\mathbb{C})$ satisfying $u=0$ on $\Sigma$ and $u(\cdot ,t)=0$, $t\in \{t_1,t_2\}$, we have
\begin{align}
&\aleph\int_Qe^{2\tau \phi}\left[\tau^3\lambda ^4\phi^3 |u|^2+\tau \lambda^2  \phi |\nabla u|^2\right]dxdt \label{C1.1}
\\
&\hskip1cm \le \int_Qe^{2\tau \phi}\left|\mathcal{L}_A^pu\right|^2dxdt +\tau \lambda \int_{\Sigma_+} e^{2\tau \phi}\phi |\partial_\nu u|^2 d\sigma dt,\nonumber
\end{align}
\end{theorem}

\subsection{Unique continuation}

An adaptation of the proof in the  case of wave equations enables us to establish the following result.

\begin{theorem}\label{theoremUCP1}
Let $u\in H^{2,1}(Q)$ satisfying $\mathcal{L}_A^pu=0$ in $Q$ and $u=0$ in $\omega \times (t_1,t_2)$, for some nonempty open subset $\omega$ of $\Omega$. Then $u=0$ in $Q$. 
\end{theorem}

We remark that Theorem \ref{theoremUCP1} can be also obtained from \cite[Proposition 3.2]{CY2017Arxiv}.

Similarly to the case of wave equations Theorem \ref{theoremUCP1} can serve to prove uniqueness of continuation from the Cauchy data on a subboundary.

\begin{corollary}\label{corollaryUCP1}
Assume that $\mathcal{L}_A^p$ is defined in $\hat{\Omega}\Supset \Omega$. Let $\Gamma_0$ be an arbitrary nonempty open subset of $\Gamma$. If $u\in H^1((t_1,t_2), H^2(\Omega))$ satisfies $\mathcal{L}_A^pu=0$ in $Q$ and $u=\partial_\nu u= 0$ in $\Sigma_0=\Gamma_0\times (t_1,t_2)$ then $u=0$ in $Q$. 
\end{corollary}

We remark that Corollary \ref{corollaryUCP1} follows also from \cite[Proposition 3.2 and Proposition 4.1]{CY2017Arxiv}.

\subsection{Final time observability inequality}

We assume in the present subsection that $t_1=0$ and $t_2=\mathfrak{t}>0$, and we consider the IBVP
\begin{equation}\label{OIpa1}
\left\{
\begin{array}{ll}
\Delta_Au-\partial_tu=f\quad \mbox{in}\; Q,
\\
u(\cdot ,0)=u_0,
\\
u_{|\Sigma}=0.
\end{array}
\right.
\end{equation}

Define the unbounded operator $\mathscr{A}:L^2(\Omega )\rightarrow L^2(\Omega )$ by
\[
\mathscr{A}u=-\Delta_Au,\quad D(\mathscr{A})=H_0^1(\Omega )\cap H^2(\Omega).
\]
It is known that $-\mathscr{A}$ generates an analytic semigroup $e^{-t\mathscr{A}}$. In particular,  for any $(u_0,f)\in L^2(\Omega )\times L^1((0,T);L^2(\Omega))$, the IBVP \eqref{OIpa1} has a unique (mild) solution $u=\mathscr{S}(u_0,f)\in C([0,T];L^2(\Omega))$ so that
\begin{equation}\label{OIp2}
\|u(\cdot ,t)\|_{L^2(\Omega)}\le \|u_0\|_{L^2(\Omega)}+\|f\|_{L^1((0,T);L^2(\Omega))},\quad 0\le t\le \mathfrak{t}.
\end{equation}
This solution is given by Duhamel's formula
\begin{equation}\label{OIp0}
u(t)=e^{-t\mathscr{A}}u_0+\int_0^te^{-(t-s)\mathscr{A}}f(s)ds,\quad 0\le t\le \mathfrak{t}.
\end{equation}
Note that if $u_0\in D(\mathscr{A})$  then $u=\mathcal{S}(u_0,0)$ satisfies 
\[
u\in C([0,\mathfrak{t}];D(\mathscr{A}))\cap C^1([0,\mathfrak{t}];L^2(\Omega )).
\]
We refer to \cite[Chapter 11]{RR1993Springer} for a concise introduction to semigroup theory.

\begin{lemma}\label{lemmaOIp1}
Let $f\in L^2(Q)$ and $\zeta \in C_0^\infty ([0,\mathfrak{t}))$. Then $u(\mathfrak{t})=\mathscr{S}(0,\zeta f)(\mathfrak{t})\in H_0^1(\Omega )$ and 
\begin{equation}\label{OIp0.1}
\|u(\mathfrak{t})\|_{H_0^1(\Omega )}\le \aleph \|\zeta f\|_{L^2(Q)},
\end{equation}
where $\aleph>0$ is a constant only depending on $\Omega$, $A$, $\mathfrak{t}$ and $\zeta$.
\end{lemma}

\begin{proof}
We have from \cite[Theorem 8.1 in page 254]{Ou2005} that $D(\mathscr{A}^{1/2})=H_0^1(\Omega)$. Therefore, in light of \eqref{OIp0}, we obtain
\[
\mathscr{A}^{1/2}u(\mathfrak{t})=\int_0^\mathfrak{t}\mathscr{A}^{1/2}e^{-(\mathfrak{t}-s)\mathscr{A}}(\zeta (s)f(s))ds.
\]
But 
\[
\left\|\mathscr{A}^{1/2}e^{-t\mathscr{A}}\right\|_{L^2(\Omega)}\le \aleph_0t^{-1/2},\quad t>0, 
\]
where the constant $\aleph_0>0$  only depends on $\Omega$, $A$. If $\mbox{supp}(\zeta)\subset [0,\mathfrak{t}-\epsilon]$, $\epsilon>0$, we find
\[
\left\|\mathscr{A}^{1/2}u(\mathfrak{t})\right\|_{L^2(\Omega)}\le \aleph_0\int_0^{\mathfrak{t}-\epsilon}(\mathfrak{t}-s)^{-1/2}\|\zeta (s)f(s)\|_{L^2(\Omega )}ds.
\]
Whence Cauchy-Schwarz's inequality yields
\begin{align*}
\left\|\mathscr{A}^{1/2}u(\mathfrak{t})\right\|_{L^2(\Omega)}&\le \aleph_0\left(\int_0^{\mathfrak{t}-\epsilon}(\mathfrak{t}-s)^{-1}ds\right)^{1/2}\|\zeta f\|_{L^2(Q)}
\\
&\le \aleph_0 \ln\left[ (\mathfrak{t}/\epsilon)^{1/2}\right]\|\zeta f\|_{L^2(Q)}.
\end{align*}
The expected inequality then follows.
\end{proof}

\begin{theorem}\label{OIp1}
Let $0\le \psi_0 \in C^4(\Omega )$ with no critical point in $\overline{\Omega}$. Set $\Gamma_+=\{x\in \Gamma ;\; \partial_{\nu_A}\psi_0(x)>0\}$ and $\Sigma_+=\Gamma_+\times (0,\mathfrak{t})$. For any $u=e^{-t\mathscr{A}}u_0$ with $u_0\in D(\mathscr{A})$ we have
\[
\|u(\mathfrak{t})\|_{H_0^1(\Omega )}\le \aleph \|\partial _\nu u\|_{L^2(\Sigma_+)},
\]
where the constant $\aleph>0$  only depends of $\Omega$, $A$ and $\mathfrak{t}$.
\end{theorem}

\begin{proof}
Let $\chi \in C_0^\infty ((0,\mathfrak{t}))$ satisfying $\chi=1$ in $[\mathfrak{t}/4,3\mathfrak{t}/4]$. As $(\Delta_A-\partial_t)(\chi u)=-\chi 'u$, we get by applying Theorem \ref{CarlemanTheoremPar2}, for any $\lambda \ge \lambda^\ast$, $\tau \ge \tau^\ast$,
\begin{align*}
&\aleph\int_Qe^{2\tau \phi}\left[\tau^3\lambda ^4\phi^3 |u|^2+\tau \lambda^2  \phi |\nabla u|^2\right]dxdt 
\\
&\hskip1cm \le \int_Qe^{2\tau \phi}\left|u\right|^2dxdt +\tau \lambda \int_{\Sigma_+} e^{2\tau \phi}\phi |\partial_\nu u|^2 d\sigma dt,
\end{align*}
where the notations are those of Theorem \ref{CarlemanTheoremPar2}.

Since the first term in the right hand side of this inequality  can be absorbed by the left hand side, provided that $\lambda$ and $\tau$ are sufficiently large, we obtain in a straightforward manner
\begin{equation}\label{OIp3}
\|u\|_{L^2(\Omega \times (\mathfrak{t}/4,3\mathfrak{t}/4))}\le \aleph \|\partial _\nu u\|_{L^2(\Sigma_+)}.
\end{equation}
Pick $\varphi \in C^\infty([0,\mathfrak{t}])$ so that that $\varphi=0$ in $[0,\mathfrak{t}/4]$ and $\varphi=1$ in $[3\mathfrak{t}/4,\mathfrak{t}]$. We easily check that $\varphi u=\mathscr{S}(0,\varphi'u)$. In light of Lemma \ref{lemmaOIp1}, we then conclude that
\[
\|u(\mathfrak{t})\|_{H_0^1(\Omega )}\le \aleph \|\varphi'u\|_{L^2(Q)}=\aleph \|\varphi'u\|_{L^2(\Omega \times (\mathfrak{t}/4,3\mathfrak{t}/4))}.
\]
This and \eqref{OIp3} imply the expected inequality.
\end{proof}

\section{Schr\"odinger equations}

Let $\phi=e^{\lambda \psi}$ be a weight function for  the Schr\"odinger operator
\[
\mathcal{L}_{A,0}^s=\Delta _A+i\partial_t
\]
and set
\[
\mathcal{L}_A^s=\Delta _A+i\partial_t+\sum_{\ell=1}^np_\ell \partial_\ell +q,
\]
where $p_1,\ldots p_n$ and $q$ belong to $L^\infty (Q ;\mathbb{C})$ and satisfy
\[
\|p_\ell\|_{L^\infty (Q)}\le \mathfrak{m},\; 1\le \ell\le n\quad \mbox{and}\quad \|q\|_{L^\infty (Q)}\le \mathfrak{m}.
\]

\subsection{Carleman inequality}

Let 
\[
\delta =\min_{\overline{Q}}|\nabla \psi|_A\; (>0)
\]
and $\mathfrak{d}=(\Omega ,t_1,t_2,\varkappa, \delta ,\mathfrak{m} )$.

\begin{theorem}\label{CarlemanTheoremSch1}
There exist three constants $\aleph=\aleph(\mathfrak{d})$, $\lambda^\ast=\lambda^\ast(\mathfrak{d})$ and $\tau^\ast =\tau^\ast(\mathfrak{d})$ so that, for any $\lambda \ge \lambda^\ast$, $\tau \ge \tau^\ast$ and $u\in H^{2,1}(Q,\mathbb{C})$, we have
\begin{align*}
&\aleph\int_Qe^{2\tau \phi}\left[\tau^3\lambda ^4\phi^3 |u|^2+\tau \lambda  \phi |\nabla u|^2\right]dxdt 
\\
&\hskip1cm \le \int_Qe^{2\tau \phi}|\mathcal{L}_A^su|^2dxdt+\int_{\partial Q} e^{2\tau \phi}\left[\tau^3\lambda ^3\phi^3|u|^2+\tau \lambda \phi |\nabla u|^2\right]d\mu \nonumber
\\
&\hskip 7cm +\int_\Sigma e^{2\tau \phi}(\tau \lambda \phi)^{-1}|\partial_tu|^2d\sigma dt ,\nonumber
\end{align*}
\end{theorem}

\begin{proof} In this proof, $\aleph$, $\lambda_j$, $\tau_j$, $j=1,2,\ldots$, denote positive generic constants only depending on $\mathfrak{d}$. 

As in the preceding section, if $\Phi=e^{-\tau \phi}$, $\tau >0$, then
\begin{align*}
&\partial_k\Phi =-\tau \partial_k\phi \Phi ,
\\
&\partial_{k\ell}\Phi=\left(-\tau \partial^2_{k\ell}\phi +\tau^2\partial_k\phi\partial_\ell\phi\right)\Phi,
\\
&\partial_t\Phi= -\tau \partial_t\phi\Phi  .
\end{align*}
We have, where $w\in H^2 (Q;\mathbb{C})$,
\[
\Phi^{-1}\Delta_A(\Phi w)=\Delta_Aw -2\tau (\nabla w|\nabla \phi)_A+\left[\tau^2  |\nabla \phi|_A^2-\tau\Delta_A\phi\right]w
\]
and
\[
i\Phi^{-1}\partial_t(\Phi w)=i\partial_tw -i\tau \partial_t\phi w.
\]
We decompose $L=\Phi^{-1}\mathcal{L}_{A,0}^s\Phi $ as follows
\[
L=L_0+L_1+c
\]
with
\begin{align*}
&L_0w=\Delta_A w+i\partial_tw+aw,
\\
&L_1w= (B|\nabla w) +bw,
\end{align*}
where we set
\begin{align*}
&a=\tau^2|\nabla \phi|_A^2,
\\
&b=-\tau \Delta_A  \phi ,
\\
&c= i\partial_t\phi,
\\
&B=-2\tau  A\nabla\phi .
\end{align*}
We have
\begin{equation}\label{sc1}
\langle L_0w|L_1w\rangle_{L^2(Q)}=\int_QL_0w\overline{L_1w}dxdt=\sum_{j=1}^6 I_j,
\end{equation}
with
\begin{align*}
&I_1=\int_Q\Delta_A w(\nabla \overline{w}|B) dxdt,
\\
&I_2=\int_Q\Delta_A wb\overline{w}dxdt,
\\
&I_3=i\int_Q\partial_tw(\nabla \overline{w}|B) dxdt,
\\
&I_4=i\int_Q\partial_twb\overline{w} dxdt,
\\
&I_5=\int_Qaw(\nabla \overline{w}|B) dxdt,
\\
&I_6=\int_Qab|w|^2dxdt.
\end{align*}
Some parts of the proof are quite similar to that  of the wave equation and therefore we omit their details. We have
\begin{equation}\label{sc2}
I_1=\int_Q(D\nabla w|\nabla \overline{w})dxdt+\int_\Sigma \left[(\nabla w|\nu)_A(\nabla \overline{w}| B)-(B/2|\nu)|\nabla w|_A^2\right]d\sigma dt,
\end{equation}
where
\[
D=C/2-(B')^t,
\]
with $C=(\mbox{div}(a_{k\ell}B))$.

Also,
\begin{align*}
\Re I_2=\int_Q\Delta_A wb\overline{w}&=-\int_Q b|\nabla w|_A^2dxdt -\Re \int_Q \overline{w}(\nabla b|\nabla w)_A dxdt 
\\
&\hskip 4cm +\Re\int_\Sigma (\nabla w|\nu )_Ab\overline{w} d\sigma dt.
\end{align*}
But $\Re (\overline{w}\nabla w)=\nabla |w|^2/2$. Therefore
\begin{align}
\Re I_2=-\int_Q b|\nabla w|_A^2dxdt &+\int_Q \Delta_A(b/2) |w|^2 dxdt\label{sc3}
\\
& -\int_\Sigma (\nabla (b/2)|\nu)_A|w|^2d\sigma dt +\Re\int_\Sigma (\nabla w|\nu )_Ab\overline{w} d\sigma dt.\nonumber
\end{align}
Let $J=\Re (I_1+I_2)$ and define
\begin{align*}
&\mathcal{A}= D-bA,
\\
&a_1=\Delta_A(b/2),
\\
&g_1(w)=\Re\left[(\nabla w|\nu)_A(\nabla \overline{w}| B) +(\nabla w|\nu )_Ab\overline{w}\right]
\\
&
\hskip 4cm -(\nabla (b/2)|\nu)_A|w|^2-(B/2|\nu)|\nabla w|_A^2.
\end{align*}
We combine \eqref{sc2} and \eqref{sc3} in order to obtain
\begin{equation}\label{sc5}
J=  \int_Q\Re (\mathcal{A}\nabla w|\nabla \overline{w})dxdt +\int_Qa_1|w|^2dxdt +\int_\Sigma g_1(w)d\sigma dt.
\end{equation}
We have once again from the calculations we done for the wave equation
\[
\mathcal{A}=2\tau A\nabla ^2\phi A+\tau \Upsilon_A(\phi(\cdot ,t)).
\]
This identity together with the following ones
\begin{align*}
&\nabla ^2\phi =\lambda^2\phi (\nabla \psi_0\otimes \nabla \psi_0)+\lambda \phi \nabla ^2\psi_0,
\\
&\Upsilon_A(\phi(\cdot ,t))=\lambda \phi \Upsilon_A(\psi_0)
\end{align*}
imply
\[
\mathcal{A}=\Theta_A(\psi_0)+2\tau \lambda^2\phi A(\nabla \psi_0\otimes \nabla \psi_0)A.
\]
As $A(\nabla \psi_0\otimes \nabla \psi_0)A$ is non negative and $\psi_0$ is $A$-pseudo-convex with constant $\kappa >0$, we get
\[
\Re (\mathcal{A}\nabla w|\nabla \overline{w})\ge \kappa \varkappa ^2|\nabla w|^2.
\]
This inequality in \eqref{sc5} yields
\begin{equation}\label{sc5.1}
J\ge  \kappa \varkappa^2\int_Q|\nabla w|^2dxdt +\int_Qa_1|w|^2dxdt +\int_\Sigma g_1(w)d\sigma dt.
\end{equation}
We find, by making an integration by parts with respect to $t$ and then with respect to $x$,
\begin{align*}
\int_Q\partial_tw(\nabla \overline{w}|B)dxdt&=\int_Q\partial_t\overline{w}(\nabla w|B)dxdt-\int_Q\partial_tw\overline{w}\mbox{div}(B)dxdt
\\
&\hskip 2cm +\int_Q(\nabla w|\partial_tB)\overline{w}dxdt
\\
&-\int_\Sigma w\partial_t\overline{w}(B|\nu)d\sigma dt+ \int_\Omega \left[w(\nabla \overline{w}|B)\right]_{t=t_1}^{t_2}dx.
\end{align*}
We then obtain, by noting that $\mbox{div}(B)=2b$,
\begin{align*}
\int_Q\partial_tw(\nabla \overline{w}|B)dxdt&=\int_Q\partial_t\overline{w}(\nabla w|B)dxdt-2\int_Q\partial_tw\overline{w}bdxdt
\\
&\hskip 2cm +\int_Q(\nabla w|\partial_tB)\overline{w}dxdt
\\
&-\int_\Sigma w\partial_t\overline{w}(B|\nu)d\sigma dt+ \int_\Omega \left[w(\nabla \overline{w}|B)\right]_{t=t_1}^{t_2}dx.
\end{align*}
From the identity 
\[
\partial_tw(\nabla \overline{w}|B)-\partial_t\overline{w}(\nabla w|B)=2i\Im[\partial_tw(\nabla \overline{w}|B)]
\]
we deduce that
\begin{align*}
2i\Im\int_Q\partial_tw(\nabla \overline{w}|B) dxdt&= -2\int_Q\partial_tw\overline{w}bdxdt+\int_Q(\nabla w|\partial_tB)\overline{w}dxdt
\\
&-\int_\Sigma w\partial_t\overline{w}(B|\nu)d\sigma dt+ \int_\Omega [w(\nabla \overline{w}|B)]_{t=t_1}^{t_2}dx.
\end{align*}
Or equivalently
\begin{align*}
-\Im\int_Q\partial_tw(\nabla \overline{w}|B) dxdt&= -\Re\left(i\int_Q\partial_tw\overline{w}bdxdt\right)-\Im \int_Q(\nabla w|\partial_tB/2)\overline{w}dxdt
\\
&+\Im\int_\Sigma w\partial_t\overline{w}(B/2|\nu)d\sigma dt-\Im \int_\Omega [w(\nabla \overline{w}|B/2)]_{t=t_1}^{t_2}dx.
\end{align*}
Observing that
\[
\Re I_3=-\Im\int_Q\partial_tw(\nabla \overline{w}|B) dxdt\quad \mbox{and}\quad I_4= i\int_Q\partial_tw\overline{w}bdxdt,
\]
we obtain
\begin{equation}\label{sc6}
\Re (I_3+I_4)= -\Im \int_Q(\nabla w|\partial_tB/2)\overline{w}dxdt +\int_\Sigma g_2(w)d\sigma dt+\int_\Omega h(w)dx, 
\end{equation}
with
\begin{align*}
&g_2(w)=\Im (w\partial_t\overline{w}(B/2|\nu)),
\\
&h(w)=-\Im\left( \left[w(\nabla \overline{w}|B/2)\right]_{t=t_1}^{t_2}\right).
\end{align*}
We find, by using once again the identity $\Re w\nabla \overline{w}=\nabla |w|^2/2$,
\[
\Re I_5=-\int_Q \mbox{div}(aB/2)|w|^2dxdt+\int_\Sigma a(B/2|\nu)|w|^2d\sigma dt
\]
and hence
\begin{equation}\label{sc7}
\Re (I_5+I_6)=\int_Qa_2|w|^2dxdt+\int_\Sigma g_3(w)d\sigma dt,
\end{equation}
with
\begin{align*}
&a_2=-\mbox{div}(aB/2)+ab,
\\
&g_3(w)= a(B/2|\nu)|w|^2.
\end{align*}
Let
\[
\mathfrak{a}=a_1+a_2=-\mbox{div}(aB/2)+ab+\Delta_A(b/2).
\]
We can carry out the same calculations as for the wave equation in order to obtain
\[
\mathfrak{a}\ge \tau ^3\lambda^4\phi ^3\delta ^4,\quad \lambda \ge \lambda_1,\; \tau \ge \tau_1.
\]
We end up getting, by combining \eqref{sc1}, \eqref{sc5.1}, \eqref{sc6} and \eqref{sc7}, the following inequality
\begin{align}
\Re \langle L_0w|L_1w\rangle_{L^2(Q)}&\ge \tau \lambda \kappa  \int_Q|\nabla w|^2dxdt +\tau ^3\lambda^4\delta ^4\int_Q\phi ^3|w|^2dxdt \label{sc8}
\\
&-\Im \int_Q(\nabla w|\partial_tB/2)\overline{w}dxdt +\int_\Sigma g(w)d\sigma dt+\int_\Omega h(w)dx,\nonumber
\end{align}
where we set $g=g_1+g_2$.

Let $\epsilon >0$. Then an elementary convexity inequality yields
\[
|(\nabla w|\partial_tB)\overline{w}|\le \aleph \phi (\epsilon \tau \lambda[\nabla w|^2+\epsilon^{-1}\tau \lambda ^2|w|^2).
\]
In consequence the third term in \eqref{sc8} can be absorbed by the first two ones, provided that $\lambda \ge \lambda_2$ and $\tau \ge \tau_2$. That is we have
\begin{align*}
\Re \langle L_0w|L_1w\rangle_{L^2(Q)}&\ge \tau \lambda \kappa  \int_Q|\nabla w|^2dxdt +\tau ^3\lambda^4\delta ^4\int_Q\phi ^3|w|^2dxdt 
\\
&+\int_\Sigma g(w)d\sigma dt+\int_\Omega h(w)dx.\nonumber
\end{align*}
The rest of the proof is almost similar to that of the wave equation.
\end{proof}

Recall that $\Gamma_+=\Gamma_+^{\psi_0}=\{x\in \Gamma ;\; \partial_{\nu_A}\psi_0>0\}$ and $\Sigma_+=\Sigma_+^\psi=\Gamma_+\times (t_1,t_2)$. As for the wave equation we have 

\begin{theorem}\label{CarlemanTheoremSch2}
There exist three constants $\aleph=\aleph(\mathfrak{d})$, $\lambda^\ast=\lambda^\ast(\mathfrak{d})$ and $\tau^\ast =\tau^\ast(\mathfrak{d})$ so that, for any $\lambda \ge \lambda^\ast$, $\tau \ge \tau^\ast$ and $u\in H^{2,1}(Q,\mathbb{C})$ satisfying $u=0$ on $\Sigma$ and $u(\cdot ,t)=0$, $t\in \{t_1,t_2\}$, we have
\begin{align}
&\aleph\int_Qe^{2\tau \phi}\left[\tau^3\lambda ^4\phi^3 |u|^2+\tau \lambda  \phi |\nabla u|^2\right]dxdt \label{C1.1}
\\
&\hskip1cm \le \int_Qe^{2\tau \phi}\left|\mathcal{L}_A^su\right|^2dxdt +\tau \lambda \int_{\Sigma_+} e^{2\tau \phi}\phi |\partial_\nu u|^2 d\sigma dt,\nonumber
\end{align}
\end{theorem}

\subsection{Unique continuation}

In this subsection, $t_1=-\mathfrak{t}$ and $t_2=\mathfrak{t}$, where $\mathfrak{t}>0$ is fixed. We recall that
\begin{align*}
E_+(\tilde{x},c)=\{x=(x',x_n)\in \mathbb{R}^{n-1}\times \mathbb{R} ;\; 0\le x_n-&\tilde{x}_n<c
\\
&\mbox{and}\; x_n-\tilde{x}_n\ge |x'-\tilde{x}'|^2/c\},
\end{align*}
with $\tilde{x}\in \Omega$, $c>0$.

\begin{theorem}\label{theoremUCs1}
Suppose that $B(\tilde{x},r)\Subset \Omega$, for some $r>0$. There exists $c^\ast=c^\ast (\varkappa,\mathfrak{m})$ with the property that, for any $0<c<c^\ast$, we find $0<\rho=\rho(c,\varkappa)<r$  so that if $u\in H^{2,1}(Q)$ satisfies $\mathcal{L}_A^su=$ in $Q$ and $\mbox{supp}(u(\cdot ,t))\cap  B(\tilde{x},r)\subset E_+(\tilde{x},c)$, $t\in (-\mathfrak{t},\mathfrak{t})$, then $u=0$ in $B(\tilde{x},\rho)\times (-\mathfrak{t}/2,\mathfrak{t}/2)$.
\end{theorem}

\begin{proof}
We proceed similarly to the proof of Theorem \ref{theoremUCw1}. We keep the same notations as in Theorem \ref{theoremUCw1}. Let $c^\ast =c^\ast (\varkappa,\mathfrak{m})$ defined as in Theorem \ref{theoremUCw1} and $0<c<c^\ast$. 

Fix $0<\eta <1$ and take instead of $\mathbf{Q}_j$, $j=0,1,2$, in Theorem \ref{theoremUCw1} the following sets
\begin{align*}
&\mathbf{Q}_0=[B(\tilde{x},\rho_0)\cap E_+]\times (-\eta \mathfrak{t}/2,\eta \mathfrak{t}/2),
\\
&\mathbf{Q}_1= \left\{E_+\cap\left[B(\tilde{x},r_0)\setminus \overline{B}(\tilde{x},\rho_1)\right]\right\}\times (-\mathfrak{t},\mathfrak{t}),
\\
&\mathbf{Q}_2=[B(\tilde{x},r_0)\cap E_+]\times [(-\mathfrak{t},- \eta \mathfrak{t})\cup (\eta \mathfrak{t},\mathfrak{t})].
\end{align*}
Also, the constants $c_j$, $j=0,1,2$ are  substituted by the following ones
\begin{align*}
&c_0=e^{\lambda (c^2/2-\epsilon/2-\gamma \eta^2 \mathfrak{t}^2/8+\delta)}\quad \mbox{in}\; \mathbf{Q}_0,
\\
&c_1=e^{\lambda (c^2/2-\epsilon+\delta )}\quad \mbox{in}\; \mathbf{Q}_1,
\\
&c_2=e^{\lambda (c^2/2-\gamma\eta^2 \mathfrak{t}^2/2+\delta )}\quad \mbox{in}\; \mathbf{Q}_2.
\end{align*}
Straightforward computations show that choosing $\gamma$ so that 
\[
\frac{4\epsilon}{3\eta^2\mathfrak{t}^2}<\gamma <\frac{4\epsilon}{\eta^2\mathfrak{t}^2}
\]
guarantee that $c_1<c_0$ and $c_2<c_0$. We can then mimic the last part of Theorem \ref{theoremUCw1} to derive that if $u\in H^{2,1}(Q)$ satisfies $\mathcal{L}_A^su=$ in $Q$ and $\mbox{supp}(u(\cdot ,t))\cap  B(\tilde{x},r)\subset E_+(\tilde{x},c)$, $t\in (-\mathfrak{t},\mathfrak{t})$, then $u=0$ in $B(\tilde{x},\rho_0)\times (-\eta \mathfrak{t}/2,\eta \mathfrak{t}/2)$. Since $0<\eta <1$ is chosen arbitrarily we get, as expected, $u=0$ in $B(\tilde{x},\rho_0)\times (-\mathfrak{t}/2, \mathfrak{t}/2)$.
\end{proof}

We  say that $\mathcal{L}_A^s$ has the reduced unique continuation property if, for any non empty open subset $\mathcal{O}\subset \Omega$ and for any $u\in H^{2,1}(Q)$ satisfying $\mathcal{L}_A^su=0$ in $Q$ and $u=0$ in $\mathcal{O} \times (-\mathfrak{t},\mathfrak{t})$, we must have $u=0$ in $\Omega \times (-\mathfrak{t}/2,\mathfrak{t}/2)$.

\begin{theorem}\label{theoremUCs2}
There exists a neighborhood $\mathcal{N}$ of $\mathbf{I}$ in $C^{2,1}(\overline{\Omega};\mathbb{R}^n\times\mathbb{R}^n)$ so that $\mathcal{L}_A^s$ has the reduced unique continuation property for any $A\in \mathcal{N}$.
\end{theorem}

\begin{proof}
Let $\mathcal{N}$ be the neighborhood of $\mathbf{I}$ in $C^{2,1}(\overline{\Omega},\mathbb{R}^{n\times n})$ given in Lemma \ref{lemmaPCH}. Pick $u\in H^{2,1}(Q)$ satisfying $\mathcal{L}_A^su=0$ in $Q$ and $u=0$ in $\mathcal{O} \times (-\mathfrak{t},\mathfrak{t})$ for some non empty open subset $\mathcal{O}\subset \Omega$. Define $\Omega_0$ as the maximal subdomain of $\Omega$ so that $u=0$ in $\Omega_0\times (-\mathfrak{t}/2,\mathfrak{t}/2)$. We claim that $\Omega \setminus \overline{\Omega_0}$ is empty which is sufficient to give the expected result. Indeed if $\Omega \setminus \overline{\Omega_0}$ is nonempty then we can proceed as in the proof of Theorem \ref{theoremUCw2} to derive that $u$ vanishes in $\mathcal{U}\times (-\mathfrak{t}/2,\mathfrak{t}/2)$, for some $\mathcal{U}$, a neighborhood of a point in $\partial \Omega_0\cap \Omega$. But this contradicts the maximality of $\Omega_0$.
\end{proof}

The uniqueness of continuation from  the Cauchy data on a subboundary is given in the  following corollary.

\begin{corollary}\label{corollaryUCs1}
Let $\mathcal{N}$ be as in Theorem \ref{theoremUCs2} with $\Omega$ substituted by larger domain $\hat{\Omega}\Supset \Omega$. Let $\Gamma_0$ a nonempty open subset of $\Gamma$ and $\Sigma_0=\Gamma_0\times (-\mathfrak{t},\mathfrak{t})$. For $A\in \mathcal{N}$, let $u\in H^{2,1}(Q)$ satisfying $\mathcal{L}_A^wu=0$ in $Q$ and $u=\partial_\nu u=0$ on $\Sigma_0$. Then $u=0$ in $\Omega \times (-\mathfrak{t}/2,\mathfrak{t}/2)$.
\end{corollary}

Also, the unique continuation across a $A$-pseudo-convex hypersurface is contained in the following theorem.

\begin{theorem}\label{theoremUCw3}
Let $H=\{x\in \omega ;\;\theta (x)=\theta (\tilde{x})\}$ be a $A$-pseudo-convex hypersurface defined in a neighborhood of $\tilde{x}\in \Omega$ with $\theta \in C^{3,1}(\overline{\omega})$. Then there exists $\mathcal{B}$, a neighborhood of $\tilde{x}$, so that if $u\in H^{2,1}(\omega \times (-\mathfrak{t},\mathfrak{t}))$ satisfies $\mathcal{L}_A^wu=0$ in $\omega\times(-\mathfrak{t},\mathfrak{t})$ and 
$\mbox{supp}(u(\cdot ,t))\subset H_+=\{ x\in \omega ;\; \theta (x)\ge \theta (\tilde{x})\}$, $t\in (-\mathfrak{t},\mathfrak{t})$,  then $u=0$ in $\mathcal{B}\times (-\mathfrak{t}/2,\mathfrak{t}/2)$.
\end{theorem}

\subsection{Observability inequality}

In this subsection $t_1=0$ and $t_2=\mathfrak{t}>0$.

Let $\mathscr{A}:L^2(\Omega )\rightarrow L^2(\Omega )$  be the unbounded operator introduced in the preceding section. That is
\[
\mathscr{A}u=-\Delta_Au,\quad D(\mathscr{A})=H_0^1(\Omega )\cap H^2(\Omega).
\]
It is known that $u(t)=e^{it\mathscr{A}}u_0$, $u_0\in L^2(\Omega )$ is the solution of the following IBVP
\begin{equation}\label{OIp1}
\left\{
\begin{array}{ll}
\Delta_Au+i\partial_tu=0\quad \mbox{in}\; Q,
\\
u(\cdot ,0)=u_0,
\\
u_{|\Sigma}=0.
\end{array}
\right.
\end{equation}
Furthermore, $u$ belongs to $C([0,\mathfrak{t}];D(\mathscr{A}))\cap C^1([0,\mathfrak{t}];L^2(\Omega ))$ whenever $u_0\in D(\mathscr{A})$ and,  for $0\le t\le \mathfrak{t}$, we have
\begin{equation}\label{OIs1}
\|u(\cdot ,t)\|_{L^2(\Omega)}=\|u_0\|_{L^2(\Omega)},\quad \|\nabla_A u(\cdot ,t)\|_{L^2(\Omega)}=\|\nabla_Au_0\|_{L^2(\Omega )}.
\end{equation}

\begin{theorem}\label{theoremOIs1}
Suppose that $0\le \psi_0\in C^4(\overline{\Omega})$ is $A$-pseudo-convex with constant $\kappa >0$ and let $\Gamma_+=\{x\in \Gamma ;\; \partial_{\nu_A}\psi_0(x)>0\}$. Then there exists a constant $\aleph$ only depending $\Omega$, $\mathfrak{t}$, $\varkappa$, $\kappa$ and $\Gamma_+$, so that, for any $u_0\in D(\mathscr{A})$, we have
\[
\|u_0\|_{H_0^1(\Omega)}\le \aleph \|\partial_\nu u\|_{L^2(\Sigma_+)},
\]
where $\Sigma_+=\Gamma_+\times (0,\mathfrak{t})$ and $u=e^{it\mathscr{A}}u_0$.
\end{theorem}

\begin{proof}
In light of Theorem \ref{CarlemanTheoremSch1} and identities \eqref{OIs1}, the expected inequality can be proved by modifying slightly that of the wave equation. 
\end{proof}

\begin{remark}\label{remarkS1}
{\rm
It is worth mentioning that the results for the elliptic, wave and Schr\"odinger equations can be extended to the case where $\Delta_A$ is substituted by the associated magnetic operator defined by
\[
\Delta_{A,\mathbf{b}}u=\sum_{k, \ell=1}^n(\partial_k +ib_k)a_{k\ell}(\partial_\ell +ib_\ell)u,
\]
with $\mathbf{b}=(b_1,\ldots ,b_n)\in W^{1,\infty}(\Omega ;\mathbb{R}^n)$.

Note that  $\Delta_{A,\mathbf{b}}u$ can be rewritten in the following form
\[
\Delta_{A,\mathbf{b}}u=\Delta_A u +2i(\nabla u|\mathbf{b})_A+\left(-|\mathbf{b}|_A^2+\mbox{div}(A\mathbf{b})\right)u.
\]
}
\end{remark}

\vskip .5cm
\end{document}